\documentclass[microtype]{gtpart}
\usepackage{amsmath,amssymb,amsthm}
\usepackage[mathscr]{eucal}
\usepackage{pinlabel}
\usepackage[dvipsnames]{xcolor}

\usepackage{mdwlist}

\definecolor{Green2}{HTML}{00D000}
\definecolor{Green3}{HTML}{00B000}
\definecolor{Green4}{HTML}{009000}

\title{Crosscap numbers of alternating knots via unknotting splices}
\author{Thomas Kindred}
\givenname{Thomas}
\surname{Kindred}
\address{Department of Mathematics, University of Nebraska, Lincoln, Nebraska 68588-0130, USA} 
\email{thomas.kindred@unl.edu}
\urladdr{www.thomaskindred.com}



\newtheorem{theorem}{Theorem}[section]
\newtheorem{lemma}[theorem]{Lemma}
\newtheorem{cor}[theorem]{Corollary}
\newtheorem{prop}[theorem]{Proposition}
\newtheorem{obs}[theorem]{Observation}
\newtheorem{question}[theorem]{Question}

\begin{document}

\begin{abstract}
Ito-Takimura recently defined a splice-unknotting number $u^-(D)$ for knot diagrams. They proved that this number provides an upper bound for the crosscap number of any prime knot, asking whether equality holds in the alternating case.  We answer their question in the affirmative. (Ito has independently proven the same result.) As an application, we compute the crosscap numbers of all prime alternating knots through at least 13 crossings, using Gauss codes.
\end{abstract}

\maketitle

\section{Introduction}\label{S:Intro}

Let $K\subset S^3$ be a knot. An embedded, compact, connected surface $F\subset S^3$ is said to {\it span} $K$ if $\partial F=K$.  The {\it crosscap number} of $K$, denoted $cc(K)$, is the smallest value of $\beta_1(F)$ among all 1-sided spanning surfaces for $K$.\footnote{Since $S^3$ is orientable, a spanning surface is 1-sided if and only if it contains a mobius band.}%
\footnote{
$\beta_1(F)=\text{rank}(H_1(F))=1-\chi(F)$ 
counts how many holes are in $F$.}

A theorem of Adams and the author \cite{ak13} states that, given an alternating diagram $D$ of a knot $K$, the crosscap number of $K$ is realized by some state surface from $D$. (Section \ref{S:Back} reviews background.) Moreover, given such $D$ and $K$, an algorithm in \cite{ak13} finds a 1-sided state surface $F$ from $D$ with $\beta_1(F)=cc(K)$.

Ito-Takimura recently introduced a splice-unknotting number $u^-(D)$ for knot diagrams.  Minimizing this number across all diagrams of a given knot $K$ defines a knot invariant, $u^-(K)$. After proving that $u^-(D)\geq cc(K)$ holds for any diagram $D$ of any nontrivial knot $K$, Ito-Takimura ask whether this inequality is ever strict in the case of prime alternating diagrams. The main theorem of this paper answers their question in the negative, and states that $u^-(D)$ is minimal among all diagrams of $K$:

\begin{theorem}\label{T:main}
If $D$ is an alternating diagram of a prime knot $K$, then
\[u^-(D)=u^-(K)=cc(K).\]
\end{theorem}

The main idea behind Theorem \ref{T:main} is that, when $D$ is alternating, each splice-unknotting sequence that realizes $u^-(D)$ corresponds to a sequence of cuts (at vertical crossing arcs) which reduces some minimal-complexity state surface to a disk, via 1-sided spanning surfaces for other knots.  The main difficulty in the proof is that for some diagrams, like the one in Figure \ref{Fi:910}, any such sequence will include non-prime diagrams.  The trouble this presents is that $u^-(D)$ is additive under diagrammatic connect sum, whereas crosscap number is not additive under connect sum. Addressing this issue requires some work. Lemmas addressing tangles appear in \textsection\ref{S:tangle}, with further technical lemmas in \textsection\ref{S:3}.  The proof of Theorem \ref{T:main} follows in \textsection\ref{S:Main}. 
Ito has independently proven the same result \cite{ito19,it19}.  

\begin{figure}[b!]
\begin{center}
\includegraphics[height=100pt]{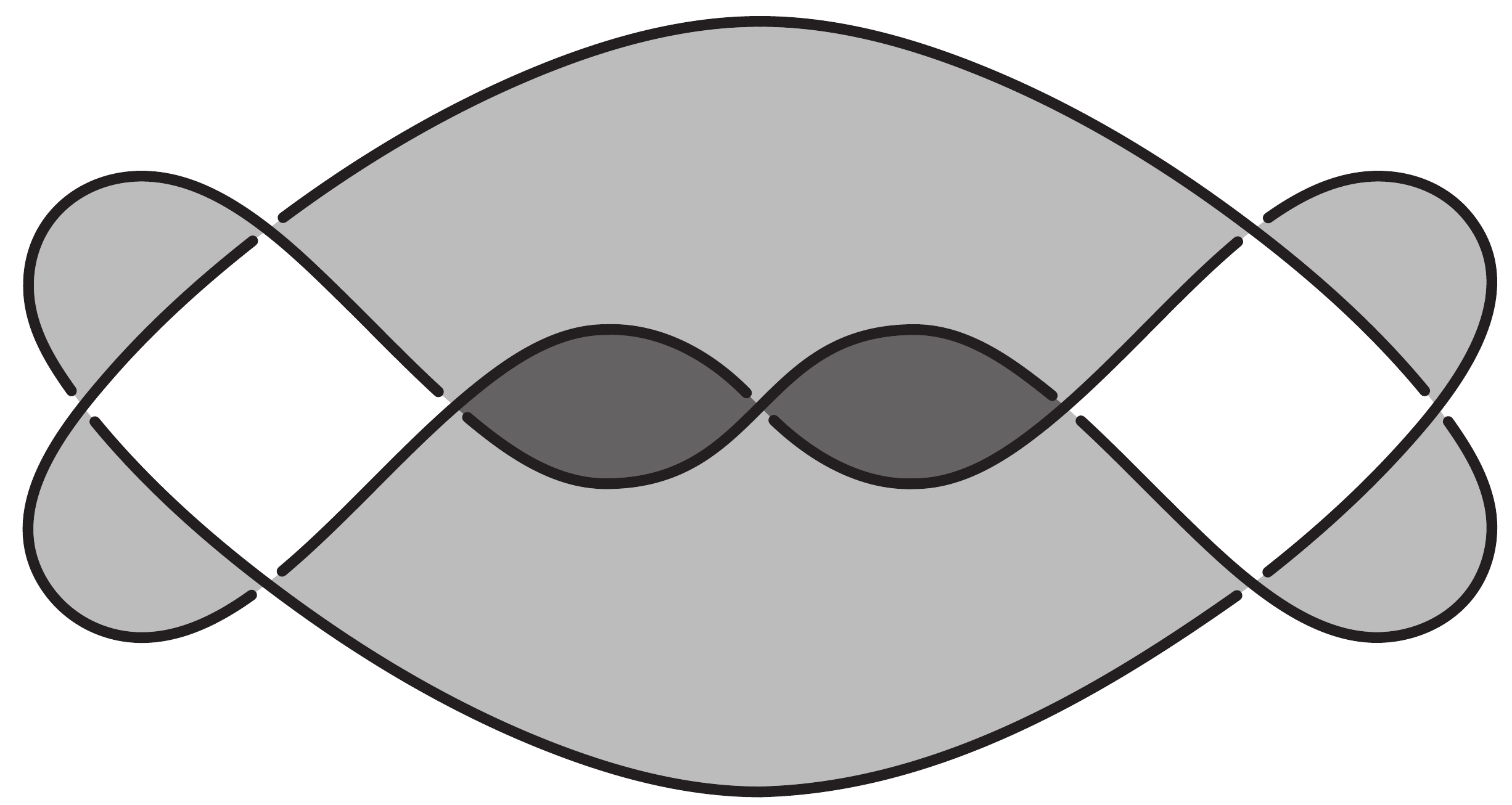}
\caption{This state surface for the $9_{10}$ knot realizes crosscap number, but cutting it at any crossing produces a state surface for either a 2-component link or a non-prime knot.}
\label{Fi:910}
\end{center}
\end{figure}

Section \ref{S:Compute} describes how Theorem \ref{T:main} enables an efficient computation of crosscap numbers for the table of prime alternating knots, using Gauss codes and data from the faces determined by the associated knot diagrams. An appendix lists the crosscap numbers for prime alternating knots through 12 crossings.\footnote{Crosscap numbers for prime alternating knots through at least 13 crossings are posted at \cite{tkcom}, together with data regarding these knots and their diagrams.} 
Previously, \cite{ak13} determined all of these values in theory, listing them through 10 crossings.  Currently, knotinfo lists crosscap numbers for 174 of the 367 prime alternating knots with 11 crossings and for 316 of the 1288 with 12 crossings \cite{kinfo}. Most of these values, and the upper and lower bounds for the remaining 11- and 12-crossing knots, come from either Burton-Ozlen, using normal surfaces \cite{bo12}, or from Kalfagianni-Lee, using properties of the colored Jones polynomial \cite{kl16}. Interestingly, every new crosscap number we compute through 12 crossings matches the upper bound currently given on knotinfo.

\section{Background}\label{S:Back}

\subsection{Splices, smoothings, and states}\label{S:State}

Let $D\subset S^2$ be an $n$-crossing diagram of a knot $K\subset S^3$. Let $c$ be a crossing of $D$, and let $\nu c$ be a disk about $c$ in $S^2$ such that $D\cap \nu c$ consists of two arcs which cross only at $c$. Up to isotopy, there are two ways to get an $(n-1)$-crossing knot diagram by replacing these two arcs within $\nu c$ with a pair of disjoint arcs.  These two replacements are called the {\it splices} of $D$ at $c$:
\[\raisebox{-.1in}{\includegraphics[height=.25in]{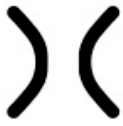}}\longleftarrow\raisebox{-.1in}{\includegraphics[height=.25in]{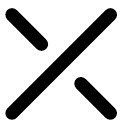}}\longrightarrow\raisebox{-.1in}{\includegraphics[height=.25in]{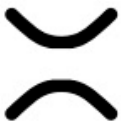}}\]

Orient $D$ arbitarily. Of the two splices of $D$ at a given crossing, one respects the orientation on $D$ and yields a diagram of a two-component link; this splice is said to be of {\it Seifert type}. The other splice yields a knot diagram and does not respect orientation.  If this {\it non-Seifert-type} splice has the same effect as a Reidemeister-I move, it is said to have {\it type $RI^-$}; otherwise this splice has {\it type $u^-$}. Note that splice types are independent of which orientation is chosen for $D$. See Figure \ref{Fi:SpliceTypes}.

There are also two {\it smoothings} of $D$ at any crossing $c$: these are the same as the splices of $D$ at $c$, except with an extra $\color{gray}A\color{black}$- or $B$-labeled arc in $\nu c$ glued to the resulting diagram:
\[\raisebox{-.1in}{\includegraphics[height=.25in]{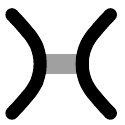}}\longleftarrow\raisebox{-.1in}{\includegraphics[height=.25in]{Crossing.pdf}}\longrightarrow\raisebox{-.1in}{\includegraphics[height=.25in]{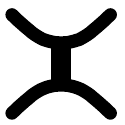}}\]
There are $2^n$ ways to smooth all the crossings in $D$, each of which results in a diagram $x$ called a {\it state}.  A state thus consists of a disjoint union of simple closed curves joined by $\color{gray}A\color{black}$- and $B$-
labeled arcs, one arc from each crossing in $D$.  The arcs and circles in $x$ are called {\it state arcs} and {\it state circles}, respectively.

\begin{figure}[b!]
\labellist
\small\hair 4pt
\pinlabel {$u^-$} [cb] at 660 90
\pinlabel {$\text{RI}^-$} [ct] at 660 60
\pinlabel {$S^-$} [cb] at 310 90
\endlabellist
\centerline{\includegraphics[width=5in]{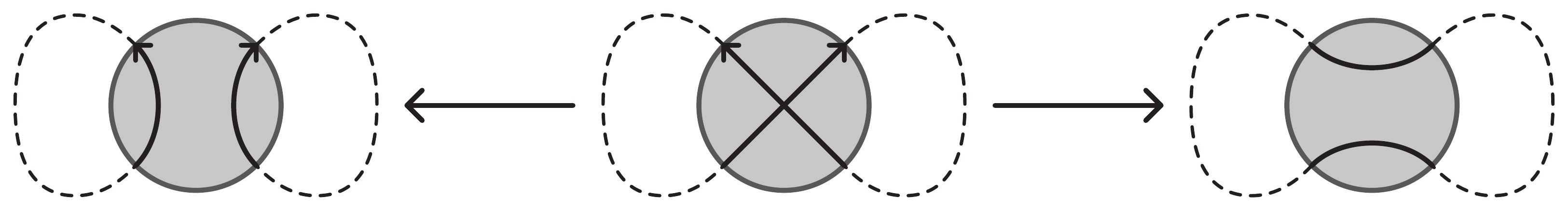}}
\caption{Seifert ($S^-$) and non-Seifert ($u^-$ and $\text{RI}^-$) type splices}
\label{Fi:SpliceTypes}
\end{figure}

\begin{figure}[b!]
\centerline{\includegraphics[width=5in]{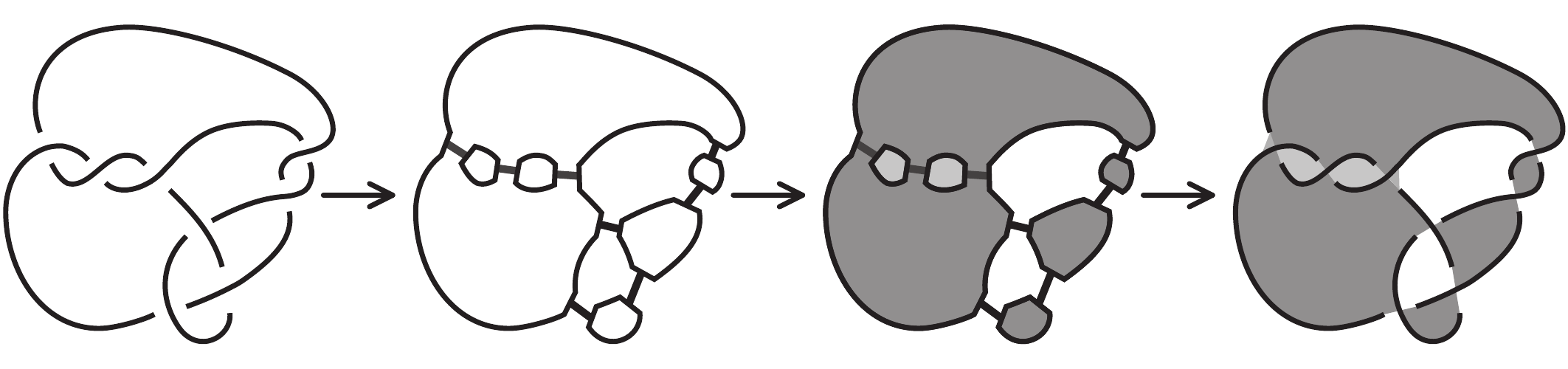}}
\caption[Constructing a state surface]{Constructing a state surface}
\label{Fi:StateSurfaceConstruction}
\end{figure}

\subsection{State surfaces}\label{S:StateSurface}

Given a state $x$ of a knot diagram, $D$, construct a {\it state surface} $F_x$ from $x$ as follows.  (See Figure \ref{Fi:StateSurfaceConstruction}.) 
First, as a preliminary step, perturb $D$ near each crossing point to obtain an embedding of $K$ in a thin neighborhood of $S^2$, such that projection $\pi:\nu S^2\to S^2$ sends $K$ to $D$. Note that the fiber over each crossing point $c$ contains a properly embedded arc in the knot complement; call this arc the {\it vertical arc} associated to $c$.  

Next, cap the state circles of $x$ with mutually disjoint disks whose interiors all lie on the same side of $S^2$. Then, near each state arc in $x$, glue on a half-twisted band (called a {\it crossing band}) which contains the associated vertical arc, such that the resulting surface $F_x$ spans $K$, $\partial F_x=K$.  

Given a state surface $F_x$ from a reduced\footnote{A knot diagram $D$ is {\it reduced} if every crossing is incident to four distinct disks of $S^2\backslash\backslash D$.} knot diagram, partition the vertical arcs in $F_x$ as $\mathscr{A}_x=\mathscr{A}_{x,S}\sqcup\mathscr{A}_{x,u}$, where $\mathscr{A}_{x,S}$ contains those of Seifert-type and $\mathscr{A}_{x,u}$ those of $u^-$ type.

\begin{obs}\label{O:2Sided}
Given a state surface $F_x$ from a reduced knot diagram, the following are equivalent:
\begin{enumerate}
\item The state surface $F_x$ is 2-sided.
\item The state $x$ has only Seifert-type smoothings, i.e. $\mathscr{A}_{x,u}=\varnothing$.
\item The boundary of each disk of $S^2\backslash\backslash x$ contains an even number of state arcs.\footnote{Notation: Whenever $X\subset Y$, $X\backslash\backslash Y$ denotes ``$X$-cut-along-$Y$.'' This is the metric closure of $X\setminus Y$, which is homeomorphic to $X\setminus \nu Y$, where $\nu Y$ is a regular open neighborhood of $Y$ in $X$.}
\end{enumerate}
\end{obs}

Regarding the last condition, note that the boundaries of the components of $S^2\backslash\backslash x$ give a generating set for $H_1(F_x)$, and each generator corresponds to an annulus or a mobius band in $F_x$ according to whether it contains an even number of state arcs.  

\begin{figure}[b!]
\centerline{\includegraphics[width=5in]{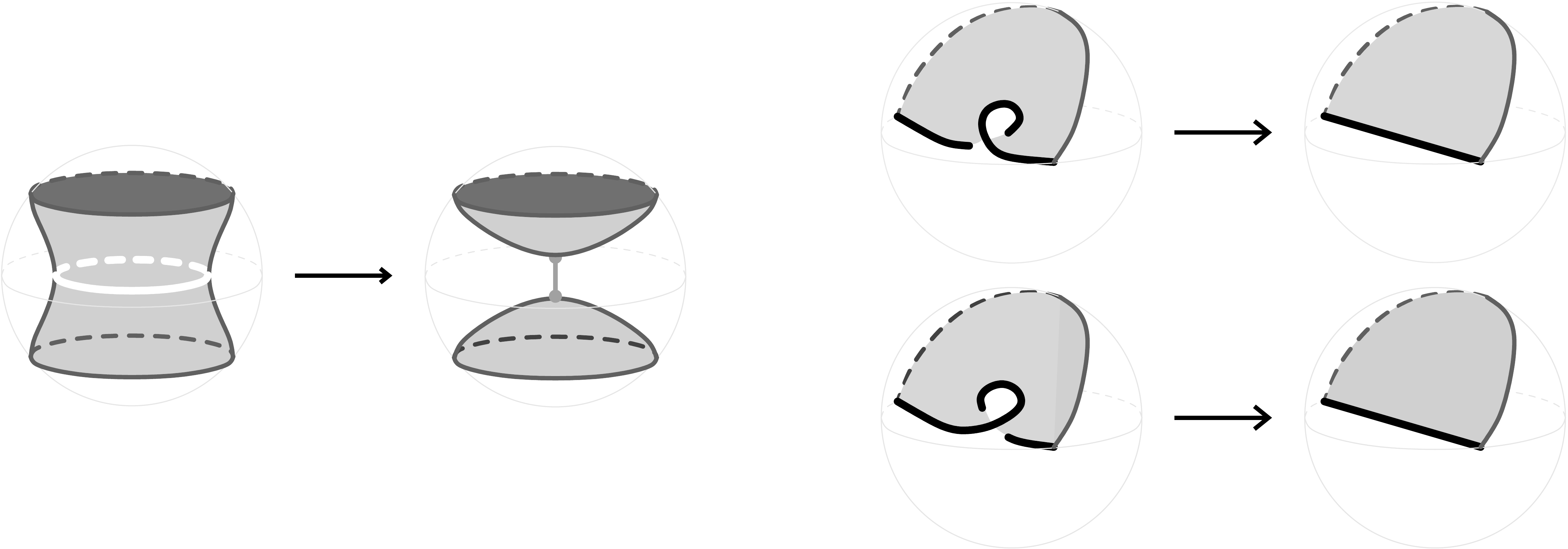}}
\caption[Compressing and $\partial$-compressing a spanning surface]{Compressing and $\partial$-compressing a spanning surface}
\label{Fi:compress}
\end{figure}

If $F$ is a spanning surface for $K$, then one can increase the complexity of $F$ by attaching a (positive or negative) crosscap or a handle.  The inverses of these local moves, called compression and $\partial$-compression, are shown in Figure \ref{Fi:compress}.  Note that attaching a $\pm$ crosscap increases $\beta_1(F)$ by 1 and changes $\text{slope}(F)$ by $\pm 2$, while attaching a handle increases $\beta_1(F)$ by 2 and does not change $\text{slope}(F)$.\footnote{When $F$ spans a knot $K$, $\text{slope}(F)$ denotes the boundary slope of $F$, which is the linking number of $K$ with a co-oriented pushoff of $K$ in $F$.}

There are two traditional notions of essentiality for spanning surfaces; we will work with the weaker, ``geometric'' notion, defined as follows.  If $F$ admits (resp. does not admit) a compression move, then $F$ is called {\it 
(in)compressible}.  If $F$ admits (resp. does not admit) a $\partial$-compression move, then $F$ is called {\it geometrically $\partial$-(in)compressible}.  If $F$ is (resp. is not) 
incompressible and 
$\partial$-incompressible, then $F$ is called {\it 
(in)essential}.\footnote{A standard application of the loop theorem implies that, with the exception of either mobius band spanning the unknot, if inclusion $\text{int}(F)\hookrightarrow S^3\setminus K$ induces an injective map on fundamental groups, then $F$ is essential.  That is, if $F$ is ``algebraically essential,'' or ``$\pi_1$-injective,'' then $F$ is (geometrically) essential. The converse is true when $F$ is 2-sided, but false in general.}

\begin{prop}\label{P:adqess}
Let $F_x$ be a 1-sided state surface from a reduced alternating diagram $D$ of a prime knot $K$, with $\beta_1(F_x)=cc(K)$.
Then the following are equivalent:
\begin{enumerate}
\item The state surface $F_x$ is 
essential.
\item The state $x$ is adequate.
\item The state $x$ has more than one non-Seifert smoothing.
\end{enumerate}
\end{prop}

\begin{proof}
Any state surface $F_x$ from an alternating diagram is a plumbing of checkerboard surfaces and is essential if and only if each checkerboard plumband is essential \cite{fkp13,fkp14,ozawa11}. Moreover, since $F_x$ comes from an alternating diagram, the checkerboard plumbands do as well, and so the checkerboard plumbands are all essential if and only if their underlying states are adequate; this is the case if and only if $x$ is adequate. Thus (1) and (2) are equivalent. 

If $x$ is non-adequate, then it differs from the Seifert state at exactly one crossing, since $\beta_1(F_x)=cc(K)$, so there is exactly one non-Seifert smoothing.  Conversely, if $x$ has at most one non-Seifert smoothing, then $x$ has exactly one non-Seifert smoothing, since $F_x$ is 1-sided. Hence, $x$ differs from the Seifert state at exactly one crossing, so $x$ is non-adequate. Thus (2) and (3) are equivalent.
\end{proof}

The main theorem in \cite{ak13} states that, when a knot $K$ has an alternating diagram $D$, the state surfaces from $D$, stabilized with crosscaps and handles, classify the spanning surfaces of $K$ up to homeomorphism type and boundary slope:

\begin{theorem}[Adams-Kindred \cite{ak13}]\label{T:AK}
Let $D$ be an alternating diagram of a knot $K$, and let $F$ be a spanning surface for $K$.  Then, by choosing an appropriate state surface from $D$ and attaching a (possibly empty) collection of crosscaps or handles, one can construct a spanning surface $F'$ for $K$ with the same number of sides (1 or 2) as $F$ and with $\beta_1(F')=\beta_1(F)$ and $\text{slope}(F')=\text{slope}(F)$.\footnote{Theorem \ref{T:AK} extends to alternating links, by replacing ``boundary slope'' with ``net'' or ``aggregate'' slope, which is the sum of the boundary slopes of $F$ along all the link components.}
\end{theorem}

In particular:

\begin{cor}[Adams-Kindred \cite{ak13}]\label{C:AK}
If $D$ is an alternating diagram of a nontrivial knot $K$, then $cc(K)$ is realized by a state surface from $D$. That is, $D$ has a state $x$ whose state surface $F_x$ is 1-sided with $\beta_1(F_x)=cc(K)$.\footnote{Corollary \ref{C:AK} also holds for alternating links.}
\end{cor}

Define the following invariant of any knot $K$: 
\[\beta_1(K):=\min_{\text{surfaces }F\text{ spanning }K}\beta_1(F).\]
Note that $\beta_1(K)=\min\{cc(K),2g(K)\}$, where $g(K)$ is the genus of $K$.  Note also that $\beta_1(K)<cc(K)$ if and only if $\beta_1(K)=2g(K)=cc(K)-1$, i.e. iff all of the surfaces realizing $\beta_1(K)$ are 2-sided. Moreover, $\beta_1(K_1\#K_2)=\beta_1(K_1)+\beta_1(K_2)$, by a standard argument. Therefore:

\begin{prop}\label{P:ConnSum}[Murakami-Yasuhara \cite{my95}]
Any knots $K_1,K_2$ satisfy $cc(K_1\#K_2)\leq cc(K_1)+cc(K_2)$.  Equality holds if and only if $cc(K_i)=\beta_1(K_i)$ for $i=1,2$. 
\end{prop}

\begin{cor}\label{C:ConnSum}
A knot $K=\#_{i\in I}K_i$ satisfies $cc(K)=\sum_{i\in I}cc(K_i)$ if and only if $K$ is prime or:
\begin{equation}\label{E:dagger}
\beta_1(K_i)=cc(K_i)\text{ for each }i\in I.
\tag{$\dagger$}
\end{equation}
\end{cor}

If $D$ is an $n$-crossing knot diagram, and $x$ is a state of $D$ with $\ell$ state circles, then its state surface satisfies 
\[\beta_1(F_x)=1-\chi(F_x)=1-(\ell-n)=n+1-\ell.\]
Thus, in order to compute $cc(K)$ when $K$ is alternating, it suffices to find a non-Seifert state $x$ of $D$
with a maximal number of state circles. Although there are $2^n-1$ possible states to choose from, \cite{ak13} describes an algorithm that shortens the list of potentially optimal states to at most $2^{\lfloor n/3\rfloor}$. 
Unfortunately, using this algorithm to compute the crosscap numbers of {\it all} alternating knots through a given number of crossings would require a separate computation for each distinct alternating knot. Ito-Takimura's splice-unknotting number $u^-(D)$ will improve the efficiency of this computation, at least in the case where $K$ is a {\it knot}. 

\subsection{Ito-Takimura's splice-unknotting number}
Let $D\subset S^2$ be an $n$-crossing diagram of a knot $K\subset S^3$. Ito-Takimura define the {\it splice-unknotting number} $u^-(D)$ as follows. 
Starting with $D$, there are $n!$ distinct sequences of non-Seifert splices, $D=D_n\to D_{n-1}\to\cdots\to D_{1}\to D_0=\bigcirc$, all of which terminate with the trivial diagram of the unknot.  
Ito-Takimura define $u^-(D)$ to be the minimum number of $u^-$ splices among these {\it splice-unknotting sequences}.\footnote{Since the over-under information at each crossing is immaterial in this definition, the splice-unknotting number $u^-(D)$ is most naturally defined on knot {\it projections}, rather than on knot diagrams, and indeed this is how Ito-Takimura defined it.} 
They prove:

\begin{theorem}[Ito-Takimura]\label{T:ItoLeq}
If $D$ is a diagram of a nontrivial knot $K$, then 
\[cc(K)\leq u^-(D).\]
\end{theorem}
\begin{figure}[b!]
\centerline{\includegraphics[width=5in]{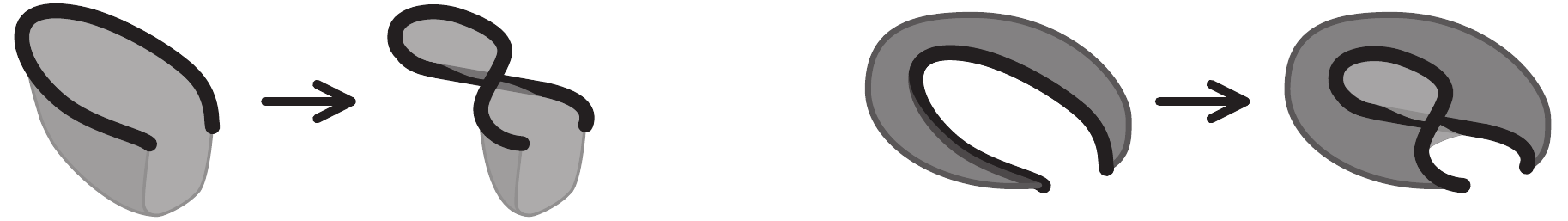}}
\caption[Ito-Takimura's construction (I)]{Ito-Takimura's construction performs an isotopy for each $R1$-splice.}
\label{Fi:ConstructR1}
\end{figure}
\begin{figure}[b!]
\centerline{\includegraphics[width=5in]{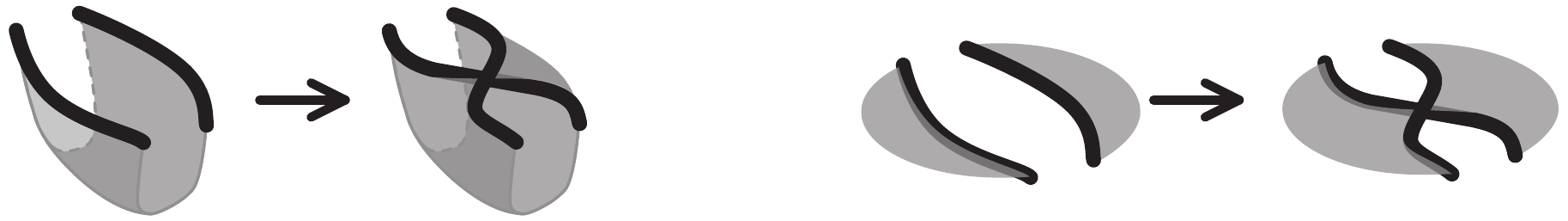}}
\caption[Ito-Takimura's construction (II)]{Ito-Takimura's construction attaches a crossing band for each $u^-$ splice.}
\label{Fi:ConstructU}
\end{figure}
The point is this: if $D=D_n\to D_{n-1}\to\cdots\to D_{}\to D_0=\bigcirc$ is a splice-unknotting sequence that realizes $u^-(D)$, then one can construct a 1-sided state surface $F_n$ for $D$ with $\beta_1(F_n)=u^-(D)$ as follows. For each $D_i$, let $K_i$ be the underlying knot. Let $F_0$ be a disk spanning the unknot $K_0$.  For each splice $D_{i}\to D_{i-1}$, construct $F_i$ from $F_{i-1}$ by:
{\begin{itemize*}
\item performing a local isotopy move, as in Figure \ref{Fi:ConstructR1}, if the splice has type $RI^-$; or
\item gluing a crossing band to $F_{i-1}$, as in Figure \ref{Fi:ConstructU}, if the splice has type $u^-$. 
\end{itemize*}}
This sequence must include at least one gluing move, or else $F_n$ would be a disk. Moreover, the first gluing move $F_{k-1}\to F_{k}$ produces a mobius band.  Thus, all surfaces $F_i$ with $i\geq k$ are 1-sided. Hence, the sequence $F_0\to\cdots\to F_n$ terminates with a 1-sided surface $F_n$ that spans $K$ and has $\beta_1(F_n)=u^-(D)$. Therefore, $cc(K)\leq \beta_1(F_n)=u^-(D)$.
 
Define the {\it splice-unknotting number} of any knot $K\subset S^3$ to be:
\[u^-(K)=\min_{\text{diagrams }D\text{ of }K}u^-(D).\] 
Observe that this is a knot invariant. Also note:

\begin{cor}\label{C:ItoLeq}
For any nontrivial knot $K$, $cc(K)\leq u^-(K)$.
\end{cor}

\begin{proof} Theorem \ref{T:ItoLeq} gives:
\[\pushQED{\qed} cc(K)\leq \min_{\text{diagrams }D\text{ of }K}u^-(D)= u^-(K)\qedhere\]
\end{proof}

Ito-Takimura prove that $u^-(D)$ is additive under diagrammatic connect sum, although crosscap number is not additive under connect sum (see Proposition \ref{P:ConnSum}). With this in mind, Ito-Takimura ask:

\begin{question}[Ito-Takimura]
Does there exist an alternating diagram $D$ of a prime knot $K$ such that $u^-(D)>cc(K)$?
\end{question}

Theorem \ref{T:main} will answer this question in the negative.

\section{Boundary connect summands and tangle subsurfaces}\label{S:tangle}
Assume throughout \textsection\ref{S:tangle}, that $D$ is an alternating diagram of a nontrivial knot $K$, and $F_x$ is a 1-sided essential state surface from $D$. Also, given a $u^-$ type vertical arc $\alpha\subset F_x$, denote $F_x\backslash\backslash\alpha=F_{x_\alpha}$ and $\partial F_{x_\alpha}=K_\alpha$.

Given a compact and connected subset $U\subset S^2$ whose boundary is disjoint from all state arcs in $x$, let $x^U$ denote the union of all state circles and state arcs of $x$ that intersect $U$, and let $F_{x}^U$ denote the associated state surface, which is a subset of $F_x$.  With this notation, we define diagrammatic notions of boundary connect sum and tangle decompositions for state surfaces, and characterize a few of their properties.

Although, strictly speaking, we will not need this fact, it is worth noting that these diagrammatic notions are more general than they seem a priori, because $D$ is alternating.  The basic point here is that, by work of Menasco \cite{men84}, any 2- or 4-punctured sphere can be isotoped in the knot complement to intersect $S^2$ in a single circle; hence, every connect sum or tangle decomposition of the alternating knot $K$ can be realized diagrammatically.  When $F_x$ is essential, every boundary connect sum or tangle decomposition of $F_x$ can also be realized diagrammatically.  For our purposes, however, it is more straightforward just to define these notions diagrammatically in the first place.

\subsection{Boundary connect summands}

A {\it boundary connect summand} of $F_x$ is any $F_x^U$, where:

\begin{itemize*}
\item each component of $\partial U$ is disjoint from state arcs and intersects $x$ transversally in two points,
\item $F_x^U$ is connected but not simply connected,
\item for any simple closed curve $\gamma\subset U$ which is disjoint from state arcs and intersects $x$ transversally in exactly two points, all of the non-nugatory state arcs in $U$ lie on the same side of $\gamma$.\footnote{A state arc $\beta$ in $x$ is {\it nugatory} if $x\setminus\text{int}(\beta)$ is disconnected.}

\end{itemize*}

Note that the last two conditions in the definition imply that any boundary connect summand $F_x^U$ is {\it prime}, meaning that if $F_x^{U'}$ is a boundary connect summand of $F_x^U$, then $F_x^{U}$ and $F_x^{U'}$ are isotopic in $F_x$.

\begin{obs}\label{O:PrimeToNot}
Suppose that $F_x$ is prime, but that, for some $u^-$ type vertical arc $\alpha$, $F_{x_\alpha}$ is not prime.  Then every boundary connect summand of $F_{x_\alpha}$ has the form $F_{x_\alpha}^U$, where $U$ is a disk or an annulus, and each component of $\partial U$ intersects the state arc $\beta=x\setminus x_\alpha$.  Moreover, when $D$ is oriented, both points of $D\cap\partial U$ where $D$ points out of $U$ lie on the same state circle, and the orientation of one of the two strands of $D\cap U$ is reversed in $D_\alpha\cap U$.
\end{obs}

\begin{figure}[b!]
\centerline{\includegraphics[width=4.5in]{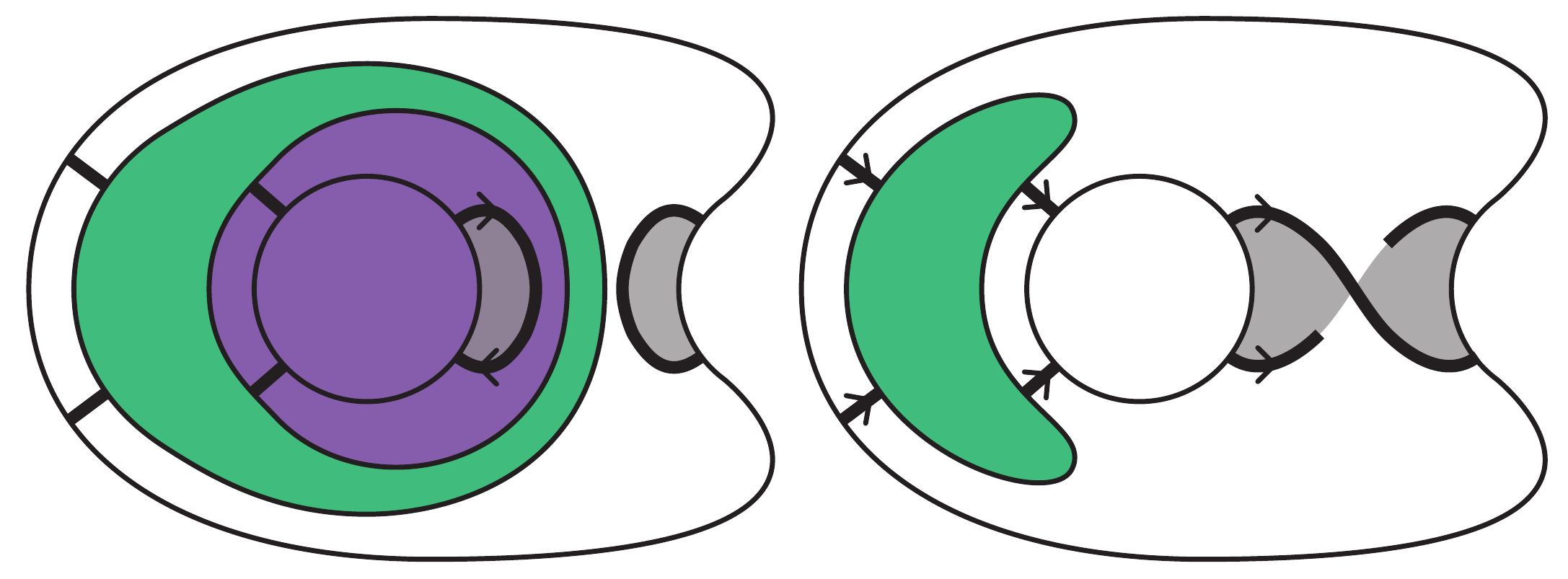}}
\caption{If $F_{x_\alpha}$ is an essential boundary connect sum, then each of its summands appears as left (purple or green).  Hence, $F_x$ has an associated minimal tangle subsurface, shown right (green).}
\label{Fi:PrimeToNot}
\end{figure}

See Figure \ref{Fi:PrimeToNot}.  In particular:

\begin{obs}\label{O:PrimeToNotKnot}
Suppose a $u^-$ type splice at a crossing $c$ in $D$ produces a diagram $D'$ of a non-prime knot $K'$.  Then there is a simple closed curve $\gamma\subset S^2$ which intersects $D$ transversally at $c$ and two other points, both on edges of $D$ not incident to $c$.  Moreover, both disks of $S^2\setminus \gamma$ contain non-nugatory crossings in $D'$.
\end{obs}


\subsection{Tangle subsurfaces}

A  {\it tangle subsurface} of $F_x$ is any $F_x^U$, where:

\begin{itemize*}
\item $U\subset S^2$ is compact and connected,
\item $\partial U$ intersects $x$ transversally in four points and is disjoint from all state arcs in $x$,
\item $F_x^U$ 
is connected but not simply connected.
\end{itemize*}
Then $F_x^U$ is the {tangle subsurface} of $F_x$ determined by $U$. Note that $D\cap U$ is a (diagrammatic) {\it tangle} in the traditional sense. 

\begin{prop}\label{P:cut1}
Suppose that $F_x^U$ is a 2-sided tangle subsurface of $F_x$ which contains a $u^-$ type vertical arc $\alpha$. If $F_{x_\alpha}^U$ is connected, then $F_{x_\alpha}$ is 1-sided.
\end{prop}

\begin{proof}
Because $\alpha\subset U$, we have:
\[F_{x_\alpha}=\left(F_{x_\alpha}^{S^2\backslash\backslash U}\right)\cup\left(F_{x_\alpha}^U\right)=\left(F_{x}^{S^2\backslash\backslash U}\right)\cup\left(F_{x}^{U}\backslash\backslash\alpha\right).\]
Thus, if $F_x^{S^2\backslash\backslash U}$ is 1-sided, the result follows immediately.  Otherwise, there exist properly embedded arcs $\rho_0\subset F_x^{S^2\backslash\backslash U}$ and $\rho_1\subset F_x^{U}$ with the same endpoints such that $\rho_0\cup\rho_1$ is the core of a mobius band in $F_x$.  Since $F_x^U\backslash\backslash\alpha$ is connected, there is a properly embedded arc $\rho_2\subset F_x^U\backslash\backslash\alpha$ such that $\rho_1\cap\rho_2=\partial\rho_1=\partial\rho_2$. The fact that $F_x^U$ is 2-sided implies that $\rho_1\cup\rho_2$ is the core of an annulus in $F_x$.  Therefore, $\rho_0\cup\rho_2$ is the core of a mobius band in $F_x\backslash\backslash\alpha$. 
\end{proof}

Say that a tangle subsurface $F_x^U$ is {\it minimal} if, for any tangle subsurface $F_x^{U'}$ with $U'\subset U$, every state arc in $U'$ is also in $U$. Note that every tangle subsurface $F_x^U$ contains a minimal one. 

\begin{obs}\label{O:PrimeToNot2}
If $F_x$ is prime and $\alpha\subset F_x$ is a $u^-$ type vertical arc such that $F_{x_\alpha}$ is essential and non-prime, then each boundary connect summand $F_{x_\alpha}^U$ of $F_{x_\alpha}$ corresponds to a minimal tangle subsurface $F_x^U$ of $F_x$. 
\end{obs}

(This extends Observation \ref{O:PrimeToNot}; see Figure \ref{Fi:PrimeToNot}.)

\begin{obs}\label{O:minimal0}
If $F_x^U$ is a minimal tangle subsurface of $F_x$, then:
\begin{itemize*}
\item no vertical arc $\alpha\subset F_x^U$ is parallel through $F_x^U$ to $\partial F_x^U$, and
\item for any properly embedded arc $\delta\subset U$ which intersects $x$ transversally in two points, both on the same state circle of $x$, all of the non-nugatory state arcs of $x$ in $U$ lie on the same side of $\delta$.
\end{itemize*}
\end{obs}


\subsection{Properties of 2-sided tangle subsurfaces}

\begin{lemma}\label{L:2SidedTangle}
Suppose that $F_x^U$ is a prime 2-sided tangle subsurface of $F_x$;
that when $D$ is oriented, both points of $D\cap\partial U$ where $D$ points out of $U$ lie on the same state circle; 
and that, for some $u^-$ type vertical arc $\alpha\subset F_x$, the orientation on one of the two strands of $D\cap U$ is reversed in $D_\alpha\cap U$.
Then $F_x^U$ contains a $u^-$ type vertical arc.
\end{lemma}

Figure \ref{Fi:2SidedTangle} illustrates the situation.

\begin{figure}[b!]
\begin{center}
\includegraphics[width=5in]{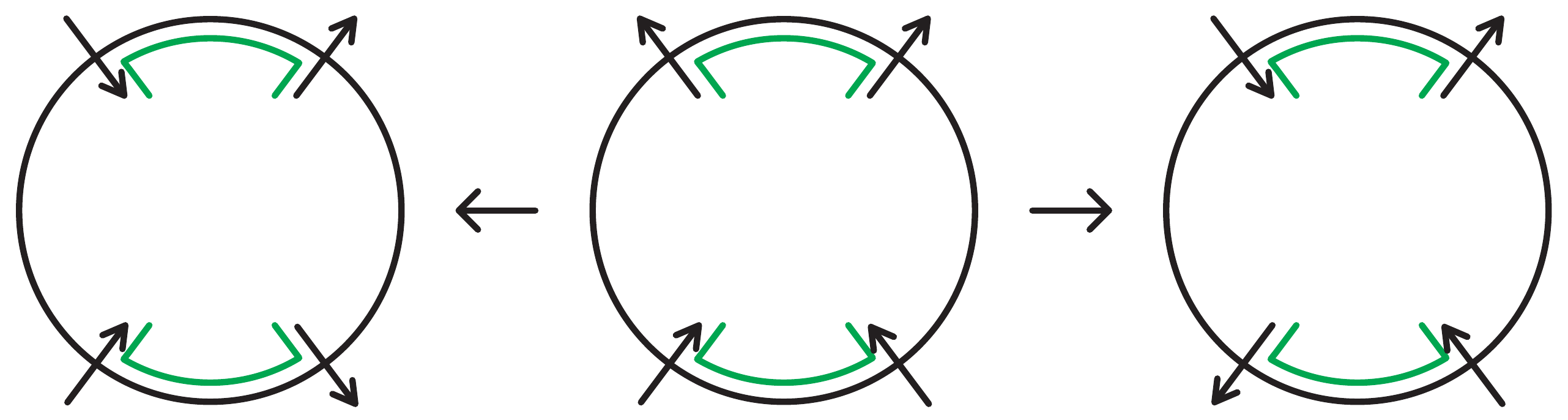}
\caption{The situation in Lemma \ref{L:2SidedTangle}: $F_x^U$ (center), the two possibilities for $F_{x_\alpha}^U$ (left, right).}\label{Fi:2SidedTangle}
\end{center}
\end{figure}

\begin{proof}
The fact that both points of $D\cap\partial U$ where $D$ points out of $U$ lie on the same state circle implies that the underlying diagrams for both $x^U$ and $x_\alpha^U$ represent {\it knots}, and that $x_\alpha^U$ is the Seifert state for its diagram. Thus, any crossing between the two strands of $D\cap U$ must have a $u^-$ type smoothing in $x^U$. 
Moreover, these two strands must cross, since $F_x^U$ is prime, in particular connected but not simply connected. 
Therefore, $F_x^U$ must contain a $u^-$ type vertical arc.
\end{proof}

In particular, using Observations \ref{O:PrimeToNot} and \ref{O:PrimeToNot2} together with Lemma \ref{L:2SidedTangle}:

\begin{cor}\label{C:2SidedTangle}
Suppose that $F_x$ is prime and $F_{x_\alpha}^U$ is a 2-sided boundary connect summand of $F_{x_\alpha}$. If necessary, adjust $U$ so that it does not intersect the state arc $\beta=x\setminus x_\alpha$ or any other state arcs that join the same two state circles that $\beta$ does. Then $F^U_x$ is a 2-sided minimal tangle subsurface in $F_x$ which contains a $u^-$ type vertical arc.
\end{cor}

\begin{lemma}\label{L:key3}
Suppose that $F_x$ contains a 2-sided minimal tangle subsurface $F_x^U$ which contains a $u^-$ type vertical arc $\alpha$. Then $F_{x_\alpha}$ is 1-sided, and $K_\alpha$ is prime.
\end{lemma}

\begin{proof}
If $F_x^U\backslash\backslash \alpha$ is connected, then $F_{x_\alpha}$ is 1-sided, by Proposition \ref{P:cut1}.  Assume instead that $F_x^U\backslash\backslash\alpha$ is not connected. Then $x_\alpha\cap U$ is not connected, so there is a properly embedded arc $\delta\subset U$ which separates the two components of $x_\alpha\cap U$. The fact that $x\cap U$ is connected implies that $|\delta\cap \beta|=1$, where $\beta$ is the state arc corresponding to $\alpha$.  
The first part of Observation \ref{O:minimal0} implies that $\alpha$ is not parallel through $F_x$ to $\partial F_x^U$. Hence, neither component of $F_x^U\setminus \alpha$ is simply connected.  Thus, each component of $x_\alpha\cap U$ contains a non-nugatory state arc.  This contradicts the second part of Observation \ref{O:minimal0}.  In all cases, therefore, $F_{x_\alpha}$ is 1-sided.
 
Assume for contradiction that $K_\alpha$ is not prime.  Then
there is a simple closed curve $\gamma\subset S^2$ which intersects $D_\alpha$ transversally in two points, neither of them crossings, such that both components of $D_\alpha\setminus\gamma$ contain non-nugatory crossings of $D_\alpha$. 
The assumption that $K$ is prime implies that $\gamma$ must intersect $\beta$. Hence, there is a properly embedded arc $\delta\subset U$ which intersects $x$ in a single point, which lies on $\beta$.  Again, the first part of Observation \ref{O:minimal0} provides non-nugatory state arcs in both components of $x_\alpha\cap U$, contradicting the second part of Observation \ref{O:minimal0}. Therefore, $K_\alpha$ is prime.
\end{proof}


\section{Technical lemmas}\label{S:3}

Throughout \textsection\ref{S:3}, $D$ will be a reduced alternating diagram of a prime knot $K$, 
and $F_x$ will be a 1-sided state surface from $D$ with $\beta_1(F_x)=cc(K)$.\footnote{Such $F_x$ exists by Theorem \ref{T:AK}; sometimes this surface will be arbitrary, subject to these conditions; other times, we will choose a particular surface $F_x$ of this type.}
Further, partitioning the vertical arcs in $F_x$ as $\mathscr{A}_{x,S}\cup\mathscr{A}_{x,u}$ as in Observation \ref{O:2Sided}, $\alpha\in \mathscr{A}_{x,u}$ will be a $u^-$ type vertical arc in $F_x$.\footnote{Such $\alpha$ exists by Observation \ref{O:2Sided}; as with $F_x$, we will sometimes take $\alpha$ to be arbitrary, and other times will we choose $\alpha$.}
We will denote $F_x\backslash\backslash\alpha=F_{x_\alpha}$ and $\partial F_{x_\alpha}=K_\alpha$.  

Note that $x=x_\alpha\cup\beta$, where $\beta\subset x$ is the state arc in that corresponds to the vertical arc $\alpha\subset F_x$, and that cutting $F_x$ at $\alpha$ corresponds to performing a $u^-$ splice on $D$ at the associated crossing. This splice yields the underlying diagram $D_\alpha$ for $x_\alpha$.  Note also that $D_\alpha$ is alternating, but not necessarily prime or reduced.

\subsection{Overview of cases}\label{S:30}
The key step in Ito-Takimura's proof that $cc(K)\leq u^-(D)$ involves building up more complex state surfaces from simpler ones, often by gluing on crossing bands in a way that corresponds to undoing a $u^-$ type splice. The key step in proving the reverse inequality is basically the opposite.  Namely, the key is to show that there exist $F_x$ and $\alpha$ such that $F_{x_\alpha}$ is 1-sided with $\beta_1(F_{x_\alpha})=cc(K_\alpha)$, such that $K_\alpha$ either is prime or satisfies the condition (\ref{E:dagger}) from Corollary \ref{C:ConnSum}.

This situation varies mainly according to whether or not $\beta_1(K)=cc(K)$. Subsection \ref{S:320} addresses the case $\beta_1(K)<cc(K)$. For each of the states $x$ which differs from the Seifert state $y$ at a single crossing, $F_x$ has a single $u^-$ type vertical arc. Also $\beta_1(F_x)=cc(K)=\beta_1(K)+1$. Lemma \ref{L:key321} establishes that, for at least one of these states $x$, $F_{x_\alpha}$ is 1-sided with $\beta_1(F_{x_\alpha})=cc(K_\alpha)$, and $K_\alpha$ is prime. 

Subsection \ref{S:3new} addresses the case $\beta_1(K)=cc(K)$. Given a 1-sided $F_x$ from $D$ with $\beta_1(F_x)=cc(K)$, Lemma \ref{L:key3} states that, if $F_x$ has a 2-sided minimal tangle subsurface which contains an arc $\alpha\in\mathscr{A}_{x,u}$, then $F_{x_\alpha}$ is 1-sided with $\beta_1(F_{x_\alpha})=cc(K_\alpha)$, and $K_\alpha$ is prime. Otherwise, every 2-sided minimal tangle subsurface in $F_x$ contains only Seifert-type vertical arcs.  (This includes the case of the knot $9_{10}$.) After some setup, this case follows easily from Corollary \ref{C:2SidedTangle}, using the condition (\ref{E:dagger}) for $K_\alpha$ and an associated condition (\ref{E:star}) for $F_{x_\alpha}$.

\subsection{Alternating knots with $\beta_1(K)<cc(K)$}\label{S:320}

In addition to the assumptions stated at the beginning of \textsection\ref{S:3}, 
assume throughout \textsection\ref{S:320} that $\beta_1(K)<cc(K)$, and that $y$ is the Seifert state of $D$. Then the associated Seifert surface satisfies $\beta_1(F_y)=\beta_1(K)=cc(K)-1>0$.

\begin{figure}[b!]
\labellist
\small\hair 4pt
\pinlabel {$y$} [c] at 155 175
\pinlabel {$z$} [c] at 640 175
\endlabellist
\centerline{\includegraphics[width=4in]{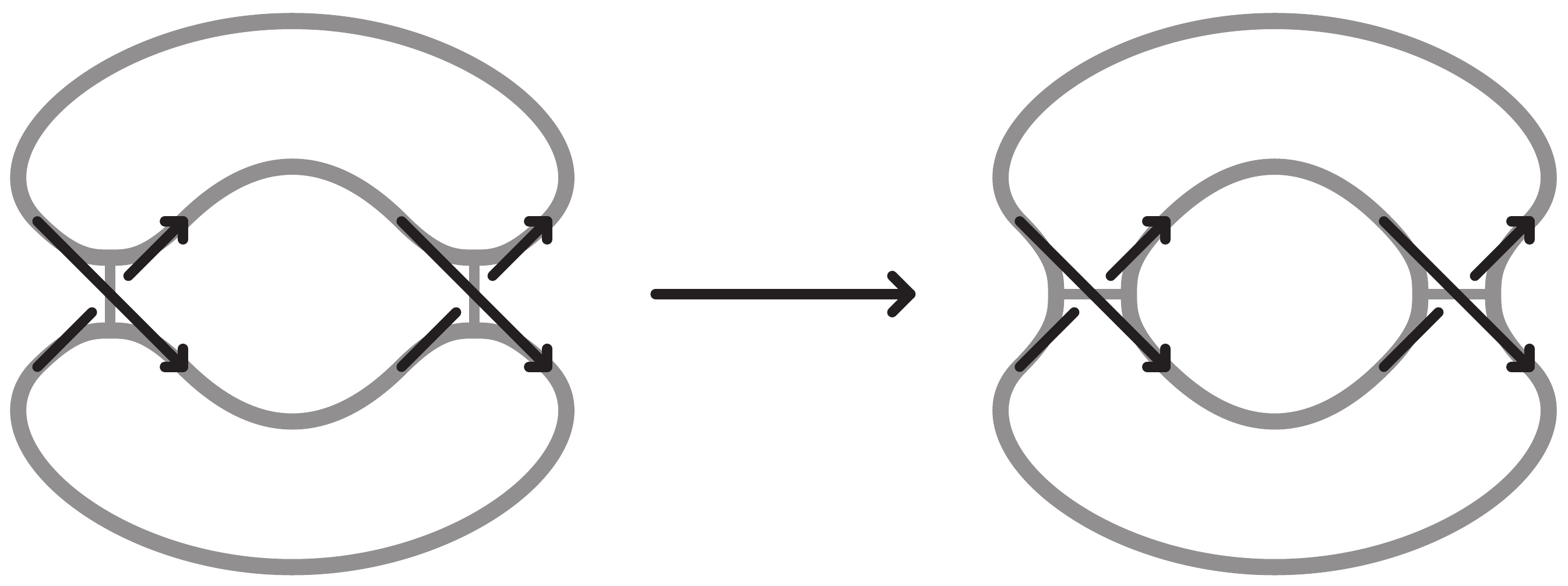}}
\caption{Proposition \ref{P:nohopf} states that if a Seifert surface $F_y$ for an alternating knot $K$ satisfies $\beta_1(F_y)<cc(K)$, then no two state arcs in $y$ join the same two state circles.}
\label{Fi:NoHopf}
\end{figure}

\begin{prop}\label{P:nohopf}
No two state arcs in $y$ join the same two state circles.
\end{prop}

\begin{proof}
If two state arcs in $y$ join the same two state circles, then reversing these two smoothings will produce a state $z\neq y$ with the same number of state circles as $y$. (See Figure \ref{Fi:NoHopf}.) But then the state surface $F_z$ will be 1-sided with $\beta_1(F_z)=\beta_1(F_y)=\beta_1(K)<cc(K)$.
\footnote{A similar argument proves more generally that if {\it any} knot $K$ satisfies $cc(K)>\beta_1(K)$, then any minimal genus Seifert surface for $K$ must have no Hopf band plumbands.}
\end{proof}

\begin{figure}[b!]
\centerline{\includegraphics[width=4.5in]{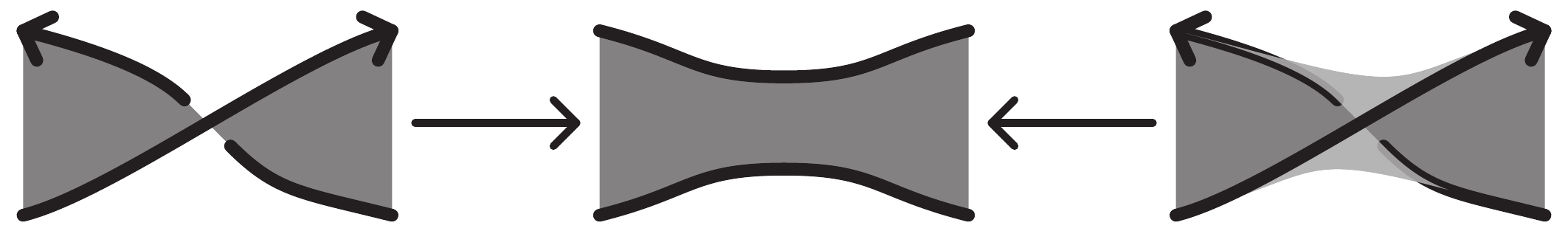}}
\caption{If $F_x$ differs from the Seifert surface $F_y$ at a single crossing $c$, then cutting $F_x$ at $c$ gives the same surface as untwisting $F_y$ at $c$.}
\label{Fi:Untwist}
\end{figure}

Reversing any one smoothing of $y$ produces a non-adequate state $x$ whose associated state surface satisfies $\beta_1(F_x)=\beta_1(F_y)+1=cc(K)$.  There is only one $u^-$ type smoothing in $x$.  Cutting $F_x$ at the associated vertical arc yields the same surface as ``untwisting'' the associated crossing band in $F_y$. See Figure \ref{Fi:Untwist}.

\begin{prop}\label{P:Untwist}
Untwisting $F_y$ at any crossing band gives a 
1-sided state surface $F_w$ from a reduced alternating knot diagram $D'$.
\end{prop}

\begin{proof}
To see that $F_w$ is 1-sided, use the fact that $D$ is reduced to obtain a simple closed curve $\gamma\subset F_y$ that passes exactly once through the given crossing band.  This $\gamma$ is the core of an annulus in $F_y$, and thus of a mobius band in $F_w$.

To see that $D'$ is reduced, suppose otherwise.  Then some state circle $v$ in $w$ either is incident to only one state arc or is incident to itself at a state arc, $\beta_1$. The former is impossible, since untwisting a crossing band merges two state circles, and all state circles in $y$ are incident to at least two crossings.  In the latter case, $v$ must be the result of merging two state circles $u_1,u_2$ from $y$ at the state arc $\beta_2$ that corresponds to the untwisted crossing band. Because no state circle in $y$ is incident to itself at a state arc, it follows that both $\beta_1$ and $\beta_2$ join $u_1$ and $u_2$. This contradicts Proposition \ref{P:nohopf}. 
%
\end{proof}

\begin{prop}\label{P:key320}
Untwisting $F_y$ at {some} crossing band yields a 
1-sided state surface $F_w$ from a {prime} reduced alternating knot diagram.  
\end{prop}

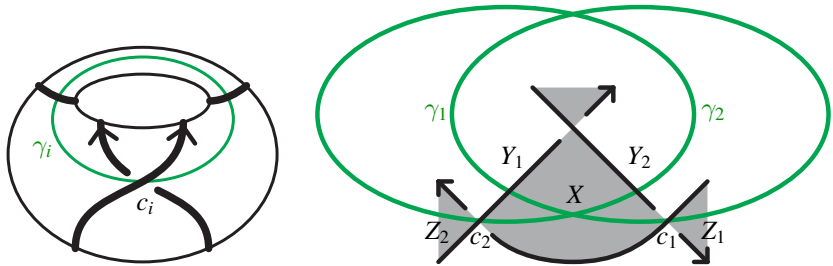
\begin{figure}[b!]
\labellist
\small\hair 4pt
\pinlabel {\color{Green4}$\gamma_i$\color{black}} [r] at 160 100
\pinlabel {$c_i$} [t] at 215 75
\pinlabel {\color{Green4}$\gamma_{1}$\color{black}} [r] at 425 125
\pinlabel {\color{Green4}$\gamma_{2}$\color{black}} [l] at 580 125
\pinlabel {$X$} [l] at 490 70
\pinlabel {$c_2$} [l] at 425 40
\pinlabel {$c_1$} [l] at 550 40
\pinlabel {$Z_2$} [l] at 395 45
\pinlabel {$Z_1$} [l] at 580 45
\pinlabel {$Y_1$} [l] at 445 95
\pinlabel {$Y_2$} [l] at 530 95
\endlabellist
\centerline{\includegraphics[width=4.5in]{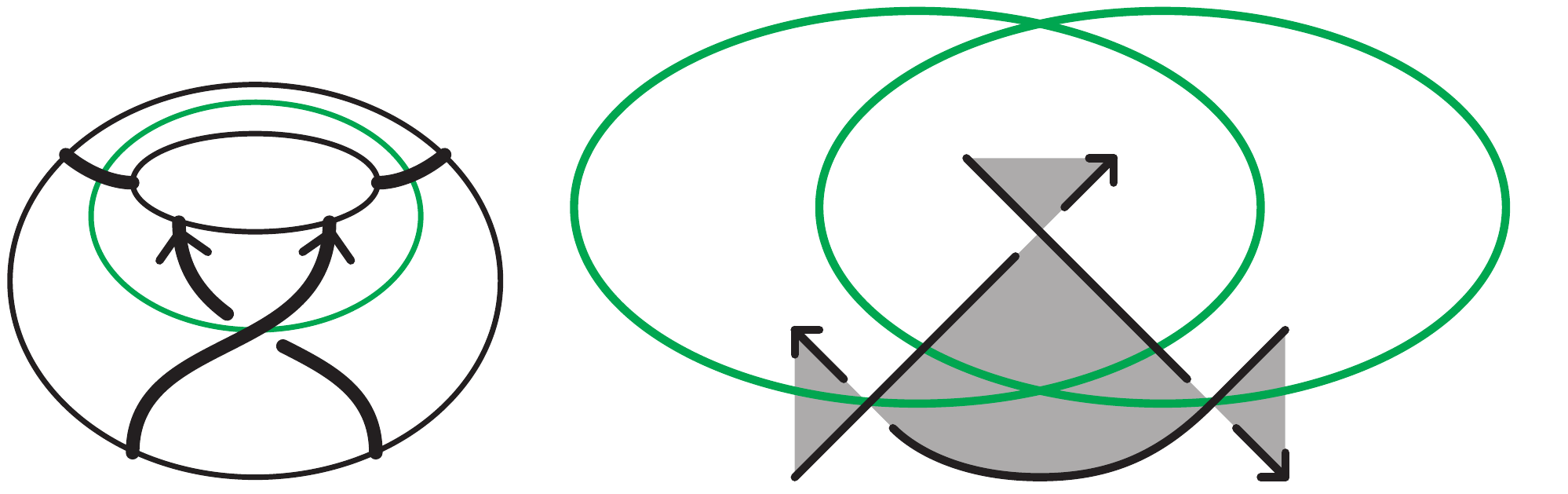}}
\caption{If $K$ is alternating and prime with $\beta_1(K)<cc(K)$, then there is a non-Seifert-type splice which yields a prime knot.}
\label{Fi:PrimeSplice}
\end{figure}

\begin{proof}
Proposition \ref{P:Untwist} implies that, for each crossing $c_i$ of $D$, untwisting $F_y$ at the crossing band near $c_i$ yields a 
1-sided state surface from a reduced 
alternating knot diagram $D_i$.
Assume for contradiction that {\it each} of these diagrams $D_i$ is non-prime.  Then Observation \ref{O:PrimeToNotKnot} implies that for every crossing $c_i$ in $D$ 
there is a simple closed curve $\gamma_i\subset S^2$ which intersects $D$ transversally at $c$ and two other points,  both of which lie on edges of $D$ which are not incident to $c$, such that $|\gamma_i\cap D'|=2$ and both disks of $S^2\setminus \gamma_i$ contain crossing points of $D_i$.  See Figure \ref{Fi:PrimeSplice}, left.  

This, together with Proposition \ref{P:nohopf} and the fact that $D$ is prime and reduced, implies that every disk of $S^2\backslash\backslash D$ is incident to {\it at least} three crossings.  Yet, an euler characteristic argument shows that some disk of $S^2\backslash\backslash D$ is incident to {\it at most} three crossings. Hence, there is a disk $X$ of $S^2\backslash\backslash D$ which is incident to {\it exactly} three crossings.  Up to symmetry, there are two possible configurations around such a disk $X$ in an arbitrary Seifert state; Proposition \ref{P:nohopf} rules out one of them. The only other possibility is that $\partial X$ is a Seifert circle of $y$, as in Figure \ref{Fi:PrimeSplice}, right.  

Let $c_1,c_2$ be two crossings on $\partial X$, and consider the arcs $\gamma_1,\gamma_2$ passing through them.
Each $\gamma_{i}$ passes through exactly three disks of $S^2\backslash\backslash D$, namely $X$ and two others, $Y_i$ and $Z_i$, where $Z_i$ is incident to $c_i$.  Since $\gamma_{1}$ and $\gamma_{2}$ intersect in a second point, outside of $X$, we must either have $Y_1=Y_2$ or $Z_1=Z_2$.  The first possibility contradicts the assumptions that $K$ is prime and $D$ is reduced; the second contradicts Proposition \ref{P:nohopf}.
\end{proof}

Therefore, with the assumptions and notation from the beginning of \textsection\ref{S:3} and \textsection\ref{S:320}:
\begin{lemma}\label{L:key321}
There exist $F_x$ and $\alpha$ such that $F_{x_\alpha}$ is 1-sided with $\beta_1(F_{x_\alpha})=\beta_1(K_\alpha)=cc(K_\alpha)$, and $D_\alpha$ is a reduced alternating diagram of the prime knot $K_\alpha$. 
\end{lemma}

\begin{proof}
Use Proposition \ref{P:key320} to obtain a state $x$ of $D$ which differs from the Seifert state $y$ of $D$ at exactly one crossing, such that untwisting $F_y$ at the associated crossing band yields a 
1-sided state surface $F_w$ from a {prime} reduced alternating knot diagram $D_\alpha$.  Then $F_x$ contains only one $u^-$ type vertical arc $\alpha$, namely the one at the crossing where $x$ differs from $y$, and $F_{x_{\alpha}}=F_w$.  Hence, $F_{x_{\alpha}}$ is a 
1-sided state surface from a prime reduced alternating knot diagram.  

To see that $\beta_1(F_{x_\alpha})=\beta_1(K_\alpha)=cc(K_\alpha)$, use Theorem \ref{T:AK} to obtain a state surface $S'$ from $D_\alpha$ with $\beta_1(S')=\beta_1(K_\alpha)$. Attaching a crossing band to $S'$ near $\alpha$ gives a 
state surface $S$ for $K$ with $\beta_1(S)=\beta_1(S')+1$.  If it were the case that $\beta_1(S')<\beta_1(F_{x_\alpha})$, then we would have the contradiction
\[\beta_1(K)=\beta_1(F_y)=\beta_1(F_{x_{\alpha}})>\beta_1(S')=\beta_1(S)+1.\] 
The fact that $F_{x_\alpha}$ is 1-sided now gives $\beta_1(F_{x_\alpha})=\beta_1(K_\alpha)=cc(K_\alpha)$. 
\end{proof}

\subsection{Alternating knots with $\beta_1(K)=cc(K)$}\label{S:3new}

In addition to the assumptions stated at the beginning of \textsection\ref{S:3}, 
assume throughout \textsection\ref{S:3new} that $\beta_1(K)=cc(K)$.

\begin{prop}\label{P:cutmin}
For any $\alpha\in\mathscr{A}_{x,u}$, $F_{x_\alpha}$ is 1-sided and essential with 
\[\beta_1(F_{x_\alpha})=\beta_1(K_\alpha)=cc(K_\alpha).\]
\end{prop}

\begin{figure}[b!]
\centerline{\includegraphics[width=5.5in]{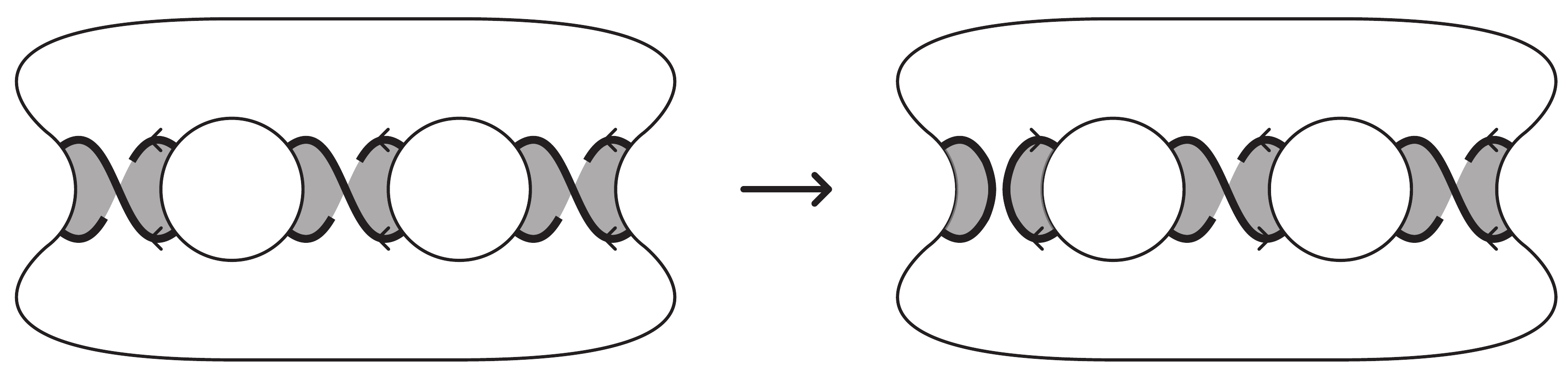}}
\caption{The situation in the proof of Proposition \ref{P:cutmin}.}
\label{Fi:TwistRegion}
\end{figure}

\begin{proof}
Assume for contradiction that some $F_{x_\alpha}$ is 2-sided. Then $x_\alpha$ is the Seifert state of $D_\alpha$ and, by Observation \ref{O:2Sided}, the boundary of each component of $S^2\backslash\backslash x_\alpha$ contains an even number of state arcs from $x_\alpha$.  Therefore, the components of $S^2\backslash\backslash x$ incident to $\alpha$ were the {\it only two} that contained an odd number of state arcs. Since $\alpha$ was arbitrary in $\mathscr{A}_{x,u}$, all state arcs in $\mathscr{A}_{x,u}$ must be incident to the same two components of $S^2\backslash\backslash x$. 

Hence, $D$ consists of $n$ crossings whose smoothing in $x$ is non-Seifert-type, together with $n$ diagrammatic tangles, each of which contains only crossings whose smoothing in $x$ is Seifert-type. (Figure \ref{Fi:TwistRegion}, left, shows the case $n=3$.)
Some of these tangles may be trivial, containing no crossings, but at least one of the tangles must contain crossings, since $\beta_1(F_x)>1$. This situation is impossible, by Lemma \ref{L:2SidedTangle}. Thus, $F_{x_\alpha}$ is 1-sided. 

Use Theorem \ref{T:AK} to obtain a state surface $S'$ from $D_\alpha$ with $\beta_1(S')=\beta_1(K_\alpha)$. Attaching a crossing band to $S'$ near $\alpha$ gives a 
state surface $S$ for $K$ with $\beta_1(S)=\beta_1(S')+1$.  If it were the case that $\beta_1(S')<\beta_1(F_{x_\alpha})$, then we would have the contradiction
\[\beta_1(K)=\beta_1(F_x)=\beta_1(F_{x_{\alpha}})+1>\beta_1(S')+1=\beta_1(S).\] 
The fact that $F_{x_\alpha}$ is 1-sided now implies that $\beta_1(F_{x_\alpha})=\beta_1(K_\alpha)=cc(K_\alpha)$, and hence that $F_{x_\alpha}$ is essential.
\end{proof}


With the setup from the start of \textsection\ref{S:3}, suppose that $F_{x_\alpha}=\natural_{i\in I}F_i$ is a boundary connect sum decomposition of $F_{x_\alpha}$ associated to the connect sum decomposition $K_\alpha=\#_{i\in I}K_i$. Say that $F_{x_\alpha}$ satisfies (\ref{E:star}) if 
\begin{equation}\label{E:star}
F_i \text{ is 1-sided with }\beta_1(F_i)=\beta_1(K_i) \text{ for each }i\in I.
\tag{$*$}
\end{equation}
\begin{obs}\label{O:star}
Any $F_{x_\alpha}$ satisfying (\ref{E:star}) is 1-sided with $\beta_1(F_{x_\alpha})=\beta_1(K_\alpha)=cc(K_\alpha)$. 
\end{obs}
Moreover, each $F_i$ is essential, as is $F_{x_\alpha}$. This further implies that the boundary connect sum decomposition of $F_{x_\alpha}$ is unique.  Note additionally that, if $F_{x_\alpha}$ satisfies (\ref{E:star}), then $K$ satisfies the property (\ref{E:dagger}) defined in Corollary \ref{C:ConnSum}.
Conversely, Theorem \ref{T:AK} implies:

\begin{obs}\label{O:daggerstar}
Any alternating knot obeying (\ref{E:dagger}) has a state surface obeying (\ref{E:star}).
\end{obs}

Here is the main result of this subsection.

\begin{lemma}\label{L:key1}
 Any 1-sided state surface $F_x$ from $D$ with $\beta_1(F_x)=\beta_1(K)$ contains a $u^-$ type vertical arc $\alpha$ such that 
$F_{x_\alpha}$ satisfies (\ref{E:star}).
\end{lemma}

\begin{proof}
%
Assume first that $F_x$ contains a 2-sided minimal tangle subsurface which contains some $\alpha\in\mathscr{A}_{x,u}$. Then Lemma \ref{L:key3} implies that $F_{x_\alpha}$ is 1-sided and $K_\alpha$ is prime. Proposition \ref{P:cutmin} further implies that $F_{x_\alpha}$ is prime with $\beta_1(F_{x_\alpha})=\beta_1(K_\alpha)=cc(K_\alpha)$. Therefore, $F_{x_\alpha}$ satisfies (\ref{E:star}). 

Assume instead that every 2-sided minimal tangle subsurface of $F_x$ contains {\it only } Seifert-type vertical arcs. Choose any $\alpha\in\mathscr{A}_{x,u}$.  If $F_{x_\alpha}$ satisfies (\ref{E:star}), then we are done. Otherwise, some boundary connect summand of $F_{x_\alpha}$ is 2-sided.  But then Corollary \ref{C:2SidedTangle} implies that the corresponding minimal tangle subsurface in $F_x$ is 2-sided and contains a $u^-$ type vertical arc, contrary to assumption.  
\end{proof}

\section{Main theorem}\label{S:Main}

Throughout \textsection\ref{S:Main}, $D$ will be a reduced alternating diagram of a nontrivial knot $K$, and $F_x$ will be a 1-sided state surface from $D$ with $\beta_1(F_x)=cc(K)$. (We no longer assume $K$ is prime.) As in \textsection\ref{S:3}, denote $\mathscr{A}_x=\mathscr{A}_{x,S}\cup\mathscr{A}_{x,u}$, 
and given $\alpha\in \mathscr{A}_{x,u}$, denote $F_x\backslash\backslash\alpha=F_{x_\alpha}$ and $\partial F_{x_\alpha}=K_\alpha$.  Now also let $F_x=\natural_{i\in I}F_i$ and $K=\#_{i\in I}K_i$ be corresponding (boundary) connect sum decompositions. Recall that $F_x$ satisfies (\ref{E:star}) if each $F_i$ is 1-sided with $\beta_1(F_i)=\beta_1(K_i)$. Recall also that, if $K$ admits such a state surface, then $K$ satisfies (\ref{E:dagger}): $cc(K_i)=\beta_1(K_i)$ for each $i\in I$. Proposition \ref{P:cutmin}  and Lemma \ref{L:key1} 
generalize to this setting as follows: 

\begin{obs}\label{O:newcut2}
For any $\alpha\in\mathscr{A}_{x,u}$, $F_{x_\alpha}$ is 1-sided and essential with $\beta_1(F_{x_\alpha})=\beta_1(K_\alpha)=cc(K_\alpha)$.
\end{obs}

\begin{obs}\label{O:key2}
If $F_x$ satisfies (\ref{E:star}), then $F_{x_\alpha}$ satisfies (\ref{E:star}) for some $\alpha\in\mathscr{A}_{x,u}$.
\end{obs}

Before moving to the main theorem, we mention an application of Observation \ref{O:key2}.  Namely, given a reduced alternating diagram $D$ of a prime alternating knot $K$ satisfying (\ref{E:star}), every 1-sided state surface $F_x$ from $D$ with $\beta_1(F_x)=\beta_1(K)$ can be obtained from a minimal splice-unknotting sequence for $D$, using the construction behind Theorem \ref{T:ItoLeq}.  Thus, a list of all minimal-length splice-unknotting sequences for $D$ conveys a list of all minimal-complexity 1-sided state surfaces from $D$. Unfortunately, the list of such sequences grows rather quickly with crossings. The data through 9 crossings is posted at \cite{tkcom}.

\begin{theorem}
Suppose that $D$ is an alternating diagram whose underlying knot $K$ is nontrivial and either is prime or satisfies (\ref{E:dagger}).  Then $u^-(D)=u^-(K)=cc(K)$.
\end{theorem}

\begin{proof}
We argue by induction on $cc(K)$. In all cases, by Theorem \ref{T:AK}, $D$ has a 1-sided state surface $F_x$ that satisfies $\beta_1(F_x)=cc(K)$.  In the base case, $F$ is a mobius band, which, cut at any crossing, becomes a disk; thus $u^-(D)=u^-(K)=1=cc(K)$.\footnote{This uses the fact that any alternating diagram of the unknot can be reduced to the trivial diagram by $RI$ moves.}  


For the inductive step, let $D$ be an alternating diagram of a knot $K$ with $cc(K)\geq 2$, where $K$ is prime or satisfies (\ref{E:dagger}). Assume that whenever $D'$ is an alternating diagram of a nontrivial knot $K'$ with $cc(K')<cc(K)$, and $K'$ is prime or satisfies (\ref{E:dagger}), then $u^-(D')=u^-(K')= cc(K')$. 

Assume first that $\beta_1(K)<cc(K)$. Then $K$ does not obey (\ref{E:dagger}), so by assumption $K$ is prime.  In this case, Lemma \ref{L:key321} provides a state surface $F_x$ and a vertical arc $\alpha\in\mathscr{A}_{x,u}$ such that $F_{x_\alpha}$ is 1-sided with $\beta_1(F_{x_\alpha})=cc(K_\alpha)$, and $K_\alpha$ is prime. Hence:
\begin{align}\label{E:Main1}
\begin{split}
cc(K)&=
\beta_1(F_x)=
\beta_1(F_{x_\alpha})+1=cc(K_\alpha)+1
=
u^-(D_\alpha)+1\\
&\geq u^-(D)\\
&\geq u^-(K).
\end{split}
\end{align}
Corollary \ref{C:ItoLeq} gives the reverse inequality, $cc(K)\leq u^-(K)$. Thus, $cc(K)=u^-(K)$. Also, $cc(K)\geq u^-(D)\geq u^-(K)$ by (\ref{E:Main1}).  Therefore, $u^-(D)=u^-(K)=cc(K)$.

Otherwise, $\beta_1(K)=cc(K)$.  Then, if $K$ is prime, $K$ satisfies (\ref{E:dagger}); also, by assumption, if $K$ is not prime, then $K$ satisfies (\ref{E:dagger}).  Thus, $K$ satisfies (\ref{E:dagger}).
Use Observation \ref{O:daggerstar} to obtain a state $x$ of $D$ such that $F_x$ satisfies (\ref{E:star}).  Then, by Observation \ref{O:key2}, there exists $\alpha\in\mathscr{A}_{x,u}$ such that $F_{x_\alpha}$ satisfies (\ref{E:star}).  Since $F_{x_\alpha}$ satisfies (\ref{E:star}), it follows that $K_\alpha$ satisfies (\ref{E:dagger}). Therefore, by repeating the computation (\ref{E:Main1}), with the subsequent application of Corollary \ref{C:ItoLeq} and squeeze argument, we can conclude in this final case that $u^-(D)=u^-(K)=cc(K)$.
\end{proof}

In particular, we have proven:

\begin{theorem}[Theorem \ref{T:main}]
If $D$ is a prime alternating diagram of a nontrivial knot $K$, then $u^-(D)=u^-(K)=cc(K)$.
\end{theorem}

\section{Computation}\label{S:Compute}

Using the fact that every prime alternating knot $K$ satisfies $u^-(K)=cc(K)$, we will construct a list $D_{\text{cc}}$ of dictionaries $D_{\text{cc}}[n]$, $n=3,4,5,\hdots$, in which to look up prime alternating knots by name and crossing number and find their crosscap numbers. Everything is coded in python. All data is available at \cite{tkcom}.  The basic idea for constructing $D_{\text{cc}}$ is this. 

First, using data imported from \cite{regina,kinfo}, we construct a list $D_G$ of dictionaries $D_G[n]$ in which to look up a prime alternating knot $K$ by name and crossing number and find a Gauss code $G=D_G[n][K]$ for a reduced alternating diagram $D$ of $K$. 

Next, we write a list $D_{\text{splice}}$ of dictionaries $D_{\text{splice}}[n]$ which associates to each $n$-crossing prime alternating knot $K$ a list of $n$ lists of knot names. For each knot $K$ the dictionary $D_G[n][K]$ provides a Gauss code, which describes a diagram $D$. Each of the $n$ lists in $D_{\text{splice}}[n][K]$ describes the connect sum decomposition of the diagram obtained $D$ by the $u^-$ type splice at one of the crossings of $D$.  

We then define a list $D_{u^-}$ of dictionaries $D_{u^-}[i]$ recursively, first setting $D_{u^-}[0][\text{\textquoteleft}1_0\text{\textquoteright}]=0$. Then for each $K$ and $n$ as above, we compute:
\[D_{u^-}[n][K]=1+\min_{i=1,\hdots,n}\sum_{j=0}^{\text{len}(D_{\text{splice}}[n][K][i])}D_{u^-}\left[\text{len}(D_{\text{splice}}[n][K][i][j])\right]\left[D_{\text{splice}}[n][K][i][j]\right]\]
Each new dictionary $D_{u^-}[n]$ records the invariant $u^-(K)$ for all prime alternating knots $K$ with $n$ crossings.  
Finally, using Theorem \ref{T:main}, we copy $D_{u^-}[i]$ for all $i\geq 3$ to construct a list $D_{\text{cc}}$ of dictionaries $D_{\text{cc}}[i]$ which record the crosscap numbers of all prime alternating knots.


The main technical challenge is that a given alternating knot can have many distinct alternating diagrams, each of which has its own unique reduced Gauss code. Thus, given a Gauss code (say, resulting from a $u^-$ type splice) its reduced form may or may not appear in $D_G$; it may not be obvious which knot the code represents. In order to solve this problem, we construct a list $D_{\text{DT}}$ of dictionaries $D_{\text{DT}}[n]$ in which to look up certain DT codes (one for each prime alternating diagram) and find the name of the associated knot. 

After some background, we give more details regarding the construction of $D_G$, $D_{\text{DT}}$, $D_{\text{splice}}$, $D_{u^-}$, and $D_{\text{cc}}$.  Of these constructions, the most computationally expensive is that of $D_{\text{DT}}$. These lists of dictionaries are among the data posted at \cite{tkcom}.

\subsection{Basics of Gauss and DT codes}\label{S:GBasics}
For an arbitrary knot diagram $D$, one obtains a Gauss code $G$ as follows.  First, choose an orientation and a starting point (away from crossings). Then, moving along $D$ accordingly, label the crossings of $D$ as $1,\hdots, n$, where $n$ is the number of crossings in $D$, according to the order in which they first appear along $D$.  Also, record all crossings of $D$, in order, as a word of length $2n$ in which each character $-n,\hdots, -1, 1,\hdots, n$ appears exactly once: the entry in the Gauss code corresponding to the overpass (resp. underpass) at the crossing with label $i$ is $i$ (resp. $-i$). Note that $D$ is reduced if and only if any Gauss code from $D$ has no cyclically consecutive entries $i$, $-i$.

Working exclusively with alternating knots and regarding mirror images as equivalent renders the signs in the Gauss code redundant. Thus, it makes sense to omit these signs, as we will do from now on. 

If $G=[c_1,c_2,\hdots,c_{2n}]$ is a Gauss code, then for each $r=1,\hdots,n$ there exist odd $i$ and even $j$ with $c_{i}=r=c_{j}$.  Thus, for each $s=1,\hdots,n$, there is a unique even integer $2\leq j(s)\leq 2n$ with $c_{j(s)}=c_{2s-1}$. The Dowker-Thistlethwaite code associated to $G$ is $[j(1),j(2),\hdots,j(n)]$. For example, the DT code abbreviating the Gauss code $[1,2,3,1,2,3]$ is $[4,6,2]$, since $c_1=1=c_4$, $c_3=3=c_6$, and $c_5=2=c_2$. The main advantage of DT codes over Gauss codes is their length; DT codes are useful when writing dictionaries.

Given a Gauss code $G$ of length $2n$, one can determine all the Gauss codes from the same diagram, but with different choices of starting point and/or orientation, by permuting and/or reversing the $2n$ characters in the Gauss code arbitrarily, and then permuting the $n$ crossing labels so that smaller labels always precede larger ones. (That is, {\it act dihedrally} on $G$ and then {\it relabel}.) Among the resulting codes, one, say $Y$, is lexicographically minimal. Call $Y$ the {\it reduced form} of $G$. Say that $G$ is {\it reduced} if its underlying diagram is reduced and if $G$ is its own reduced form.

For any reduced Gauss code $G$ which represents a prime alternating knot diagram, there is, up to isotopy and reflection, a unique knot diagram $D$ whose reduced Gauss code is $G$. (There may be several choices of basepoint and orientation on $D$ that give $G$.)

A reduced Gauss code $G$ of a knot $K$ represents a connect sum if and only if $G=w_1w_2w_3$, where $w_2$ is a nonempty proper subword of $G$ that shares no characters with $w_1$ nor $w_3$.  After relabeling (so that smaller labels always precede larger ones), $w_2$ and $w_1w_3$ give Gauss codes for two, not necessarily prime, connect summands of $K$.  Continuing in this way eventually gives the connect sum decomposition of $K$.

\begin{figure}[b!]
\labellist
\small\hair 4pt
\pinlabel {1} [c] at 310 425
\pinlabel {1} [c] at 255 427
\pinlabel {2} [c] at 200 390
\pinlabel {2} [c] at 185 345
\pinlabel {3} [c] at 225 443
\pinlabel {3} [c] at 225 480
\pinlabel {4} [c] at 360 355
\pinlabel {4} [c] at 360 390
\pinlabel {5} [c] at 520 390
\pinlabel {5} [c] at 535 345
\pinlabel {6} [c] at 495 443
\pinlabel {6} [c] at 500 480
\pinlabel {7} [c] at 412 423
\pinlabel {7} [c] at 467 423
\pinlabel {1} [c] at 850 445
\pinlabel {1} [c] at 850 405
\pinlabel {2} [c] at 780 360
\pinlabel {2} [c] at 745 375
\pinlabel {3} [c] at 772 455
\pinlabel {3} [c] at 822 464
\pinlabel {4} [c] at 900 373
\pinlabel {4} [c] at 955 373
\pinlabel {5} [c] at 1081 358
\pinlabel {5} [c] at 1114 376
\pinlabel {6} [c] at 1037 463
\pinlabel {6} [c] at 1090 455
\pinlabel {7} [c] at 1005 407
\pinlabel {7} [c] at 1005 443
\pinlabel {7} [c] at 360 155
\pinlabel {14} [c] at 374 72
\pinlabel {3} [c] at 235 72
\pinlabel {1} [c] at 230 45
\pinlabel {9} [c] at 280 -55
\pinlabel {8} [c] at 208 50
\pinlabel {2} [c] at 130 55
\pinlabel {4} [c] at 330 45
\pinlabel {12} [c] at 595 60
\pinlabel {5} [c] at 435 -55
\pinlabel {6} [c] at 517 50
\pinlabel {11} [c] at 485 70
\pinlabel {10} [c] at 385 45
\pinlabel {13} [c] at 490 45
\pinlabel {7} [c] at 910 125
\pinlabel {14} [c] at 944 99
\pinlabel {3} [c] at 832 92
\pinlabel {1} [c] at 812 18
\pinlabel {9} [c] at 830 -32
\pinlabel {8} [c] at 745 50
\pinlabel {2} [c] at 730 75
\pinlabel {4} [c] at 876 25
\pinlabel {12} [c] at 1122 80
\pinlabel {5} [c] at 1025 -28
\pinlabel {6} [c] at 1116 47
\pinlabel {11} [c] at 1023 92
\pinlabel {10} [c] at 980 25
\pinlabel {13} [c] at 1045 15
\pinlabel {Crossings around $A$-faces:} [l] at 130 295
\pinlabel {$\big[[1, 4, 7], [2, 1, 3], [5, 4, 2, 3, 6], [5, 6, 7]\big]$}  [l] at 130 265
\pinlabel {Crossings around $B$-faces:} [r] at 1160 295
\pinlabel {$\big[[1, 2, 4], [2, 3], [4, 5, 7], [5, 6], [6, 3, 1, 7]\big]$}  [r] at 1160 265
\pinlabel {Edges around $A$-faces:} [l] at 130 215
\pinlabel {$\big[[14, 4, 10], [8, 1, 3], [12, 5, 9, 2, 7], [13, 6, 11]\big]$}  [l] at 130 185
\pinlabel {Edges around $B$-faces:} [r] at 1160 215
\pinlabel {$\big[[4, 1, 9], [2, 8], [10, 5, 13], [6, 12], [11, 7, 3, 14]\big]$}  [r] at 1160 185
\endlabellist
\centerline{\includegraphics[width=2.75in]{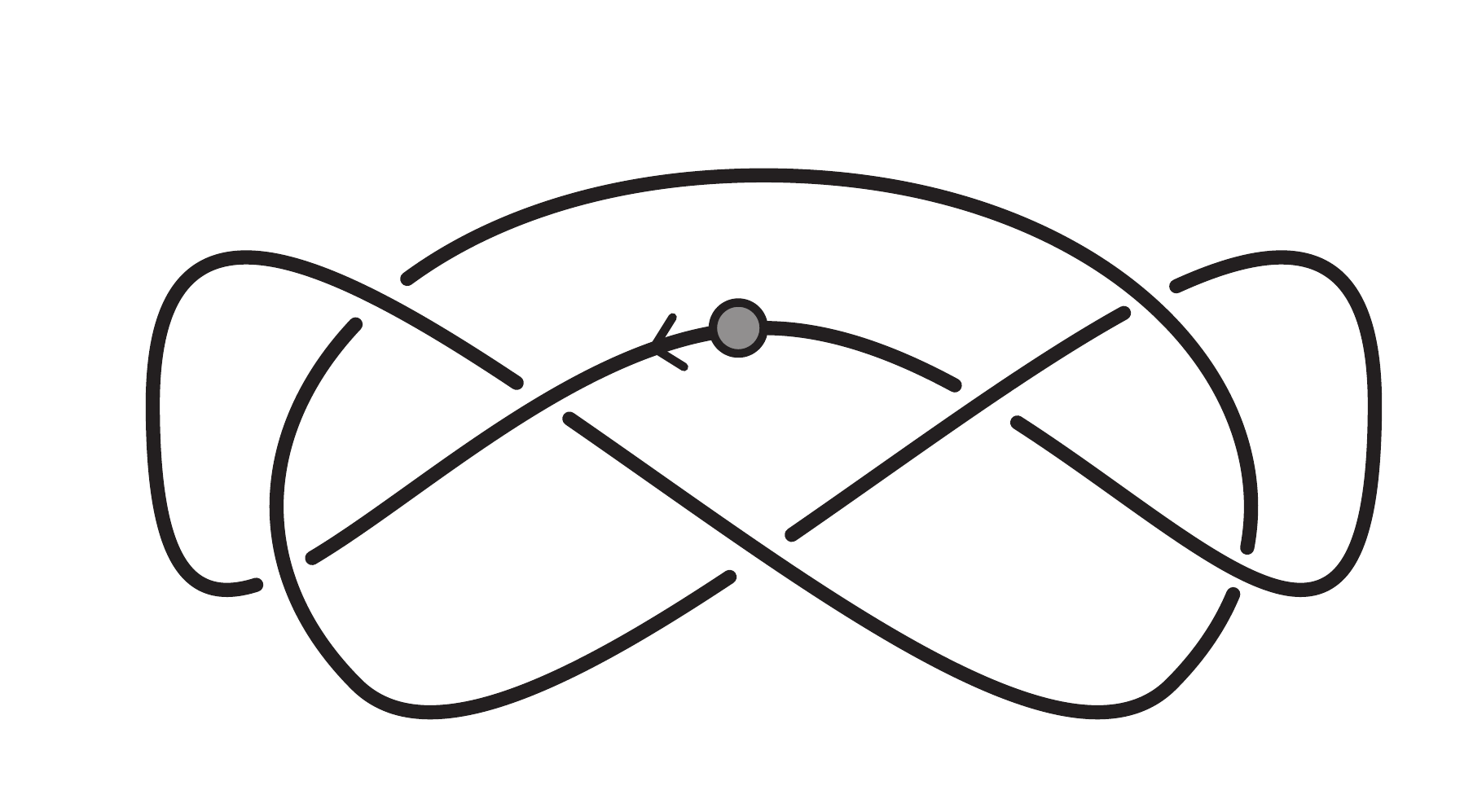}\hspace{.25in}\includegraphics[width=2.75in]{Knot77}}~\\~\\
\centerline{\includegraphics[width=2.75in]{Knot77}\hspace{.25in}\includegraphics[width=2.75in]{Knot77}}
\caption{Face data from the diagram of $7_7$ with Gauss code $[1,2,3,1,4,5,6,3,2,4,7,6,5,7]$.}
\label{Fi:FaceData}
\end{figure}

\subsection{Face data and flypes}\label{S:GFace}
We have imported Gauss codes from \cite{regina,kinfo}, one for each prime alternating knot through $n$ crossings. We we have organized this data as a list, $D_G$, of dictionaries, $D_G[n]$, so that one can look up the name (e.g. \textquoteleft$7_4$\textquoteright) of any $n$-crossing prime alternating knot $K$ in $D_G[n]$ and find a Gauss code $D_G[n][K]$ for a reduced alternating diagram of $K$.  Then we clean up this data by replacing each Gauss code with its reduced form. Finally, we augment this data by replacing each entry in each dictionary, a Gauss code $G$, with the list $[G,S]$: here, $S$ lists the signs of the crossings of the diagram associated to $G$, with the convention that the first crossing is an overpass with a positive sign. Although these signs are encoded by $G$, they take some time to compute; recording them now ensures that we only need to compute them this once.

We now set about constructing a list $D_{\text{DT}}$ of dictionaries $D_{\text{DT}}[n]$ in which to look up certain DT codes (one code for each prime alternating diagram with $n$ crossings) and find the name of the associated knot. 
The key is to find a list $D_0,\hdots,D_k$ of all reduced alternating diagrams of each prime alternating knot $K$.
To do so, we need to use the flyping theorem, conjectured by Tait \cite{tait} and proven by Menasco-Thistlethwaite \cite{mt91,mt93}. Here is how to do this.

Let $G_0$ be a reduced Gauss code of a prime alternating knot. If $G_0$ has length $2n$, then the associated projection has $n$ crossings, which are joined by $2n$ {\it edges} (in the sense that the projection is a 4-valent graph). Also, the projection cuts $S^2$ into $n+2$ {\it black} and {\it white} disks, or {\it faces}.  The {\it face data} from $G_0$ records which edges and crossings are incident to each face, proceeding counterclockwise around the boundary of the face.\footnote{For edges, the data at \cite{tkcom} also records the orientation of the edge with a sign: $+$ if the edge runs counterclockwise along the boundary of the face, $-$ if it runs clockwise.} It is convenient to partition this data into four sets, two for crossings and two for edges, each split between data from the black faces and from the white. Figure \ref{Fi:FaceData} shows an example.

This face data allows one to identify possible flype moves on the diagram. To do this, define four sets as follows. The first two sets, $EE_B$ and $EE_W$
, consist of pairs of distinct edges which lie on the boundary of the same (black or white, resp.) face and which do not share any endpoints.  The other two sets, $ECE_B$ and $ECE_W$
, consist of triples, each triple consisting of two edges and a crossing, such that neither edge is incident to the crossing and the two edges abut the (two black or two white, resp.) faces incident to the crossing.
Associate to each element of $EE_B$ ($EE_W$, resp.) an arc whose interior lies in a black (white) face of $S^2\backslash\backslash D$ and whose endpoints lie on non-incident edges of $D$. 
Likewise, associate to each element of $ECE_B$ ($ECE_W$, resp.) an arc whose interior intersects $D$ in a single point, a crossing, and otherwise lies entirely in two black (white) faces of $S^2\backslash\backslash D$, and whose endpoints lie on edges of $D$ which are not incident to this crossing.  Thus, associated to each element of $EE_B\cap ECE_W$ ($EE_W\cap ECE_B$, resp.) is a simple closed curve which intersects one black (white) face of $S^2\backslash\backslash D$ and two white (black) faces of $S^2\backslash\backslash D$, and which intersects $D$ transversally in two edges $e_1$, $e_2$ and one crossing $c$, none of them incident. 
In this way, each element of $EE_B\cap ECE_W$ identifies a possible flype move on $D$, as does each element of $EE_W\cap ECE_B$.

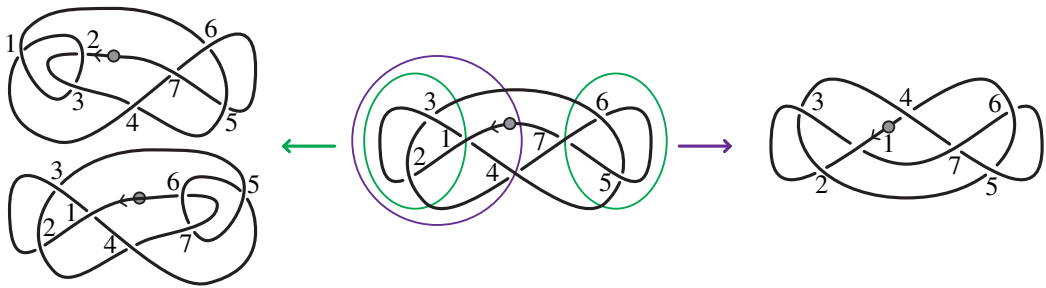
\begin{figure}[b!]
\labellist
\small\hair 4pt
\pinlabel {1} [c] at 60 480
\pinlabel {3} [c] at 168 387
\pinlabel {2} [c] at 192 487
\pinlabel {4} [c] at 255 355
\pinlabel {5} [c] at 413 355
\pinlabel {6} [c] at 380 505
\pinlabel {7} [c] at 322 408
\pinlabel {1} [c] at 157 211
\pinlabel {2} [c] at 122 187
\pinlabel {3} [c] at 134 277
\pinlabel {4} [c] at 219 160
\pinlabel {5} [c] at 448 255
\pinlabel {6} [c] at 320 260
\pinlabel {7} [c] at 340 160
\pinlabel {1} [c] at 755 329
\pinlabel {2} [c] at 713 300
\pinlabel {3} [c] at 730 388
\pinlabel {4} [c] at 830 275
\pinlabel {5} [c] at 1010 258
\pinlabel {6} [c] at 1005 388
\pinlabel {7} [c] at 905 323
\pinlabel {1} [c] at 1462 320
\pinlabel {2} [c] at 1355 255
\pinlabel {3} [c] at 1349 387
\pinlabel {4} [c] at 1490 395
\pinlabel {5} [c] at 1627 252
\pinlabel {6} [c] at 1632 385
\pinlabel {7} [c] at 1568 285
%
\endlabellist
\centerline{\includegraphics[width=5.5in]{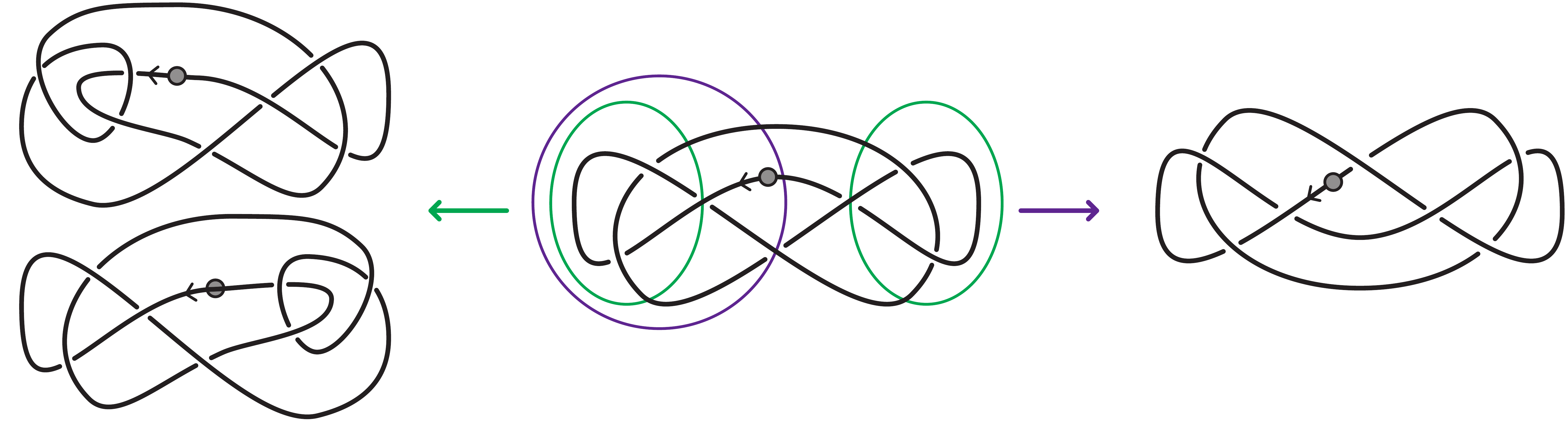}}
\caption{Four flype moves on the same diagram of $7_7$.}
\label{Fi:FlypeCode}
\end{figure}

The flype move changes the Gauss code by removing both $c$ terms, re-inserting them in the intervals of the Gauss code associated to $e_1$ and $e_2$, and then relabeling. More precisely, with $G=(c_1,\hdots,c_{2n})$, there exist indices $1\leq i_1,i_2\leq 2n-1$ such that $e_1$ joins $c_{i_1}$ and $c_{i_1+1}$, while $e_2$ joins $c_{i_2}$ and $c_{i_2+1}$. Assume without loss of generality that $i_1<i_2$. There are also two indices $1\leq j_1<j_2\leq 2n$ such that $c_{j_1}=c=c_{j_2}$. There are two explicit possibilities for the Gauss code resulting from the flype. 
If $i_1<j_1<i_2<j_2$, then the new Gauss code is 
\[(c_1,\hdots,c_{i_1},c,c_{i_1+1},\hdots,\widehat{c_{j_1}},\hdots,c_{i_2},c,c_{i_2+1},\hdots,\widehat{c_{j_2}},\hdots,c_{2n}),\] after relabeling. (The hats indicate entries to delete from the Gauss code.)
Otherwise, $j_1<i_1<j_2<i_2$, and the new Gauss code is \[(c_1,\hdots,\widehat{c_{j_1}},\hdots,c_{i_1},c,c_{i_1+1},\hdots,\widehat{c_{j_2}},\hdots,c_{i_2},c,c_{2_1+1},\hdots,c_{2n}),\] after relabeling. See Figure \ref{Fi:FlypeCode}. 
This is how we construct, for each element of $EE_B\cap ECE_W$ and $EE_W\cap ECE_B$, a Gauss code for the diagram produced by the associated flype move on $D$.

Given a Gauss code $G$ for an alternating diagram $D_0$ of a prime knot $K$, we are now ready to compute a list $L$ of $DT$ codes, one from each reduced alternating diagram of $K$. (Each $DT$ code will correspond to the reduced Gauss code of some diagram of $K$.) Begin by computing the reduced form $G_0$ of $G$, let $T_0$ be its DT code, and let $L=[T0]$.  Then compute $EE_B\cap ECE_W$ and $EE_W\cap ECE_B$ from $G_0$ to identify possible flype moves on $D_0$. Compute the reduced form of the Gauss code resulting from each flype move.  If $L$ does not already contain the DT code for this reduced Gauss code, then append that DT code.  After doing this for each possible flype move on $D_0$, repeat the process for each of the other diagrams described by the DT codes in $L$, appending any new DT codes to $L$.  The flyping theorem implies that this process will produce a list $L$ consisting of one DT code for each reduced alternating diagram of $K$.


Now we can build the dictionary $D_{\text{DT}}$: for each knot type $K$, say with Gauss code $G$, we compute the list $L$ as above from $G$, and then for each $T_i$ in $L$ we update the dictionary $D_{\text{DT}}$ with the entry $T_i:K$.
For example, for knots with seven crossings, $D_{\text{DT}}$ looks like:

\centerline{\begin{tabular}{||c|c||c|c||}\hline\hline
DT code&knot&DT code&knot\\ \hline
$[8, 10, 12, 14, 2, 4, 6]$&$7_1$&$[4, 10, 14, 12, 2, 8, 6]$&$7_2$\\
$[6, 10, 12, 14, 2, 4, 8]$&$7_3$&$[6, 12, 10, 14, 2, 4, 8]$&$7_4$\\ 
$[4, 10, 12, 14, 2, 8, 6]$&$7_5$&$[4, 10, 14, 12, 2, 6, 8]$&$7_5$\\
$[4, 8, 12, 2, 14, 6, 10]$&$7_6$&$[4, 8, 12, 10, 2, 14, 6]$&$7_6$\\
$[4, 8, 10, 12, 2, 14, 6]$&$7_7$&$[4, 8, 12, 14, 2, 6, 10]$&$7_7$\\ \hline\hline
\end{tabular}}

The dictionary list $D_{\text{DT}}$ through at least 13 crossings is available at \cite{tkcom}.

\subsection{Splices from face data}\label{S:GSplice}

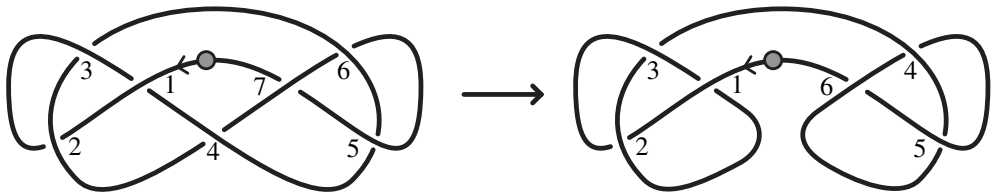
\begin{figure}[b!]
\labellist 
\small\hair 4pt
\pinlabel {1} [c] at 818 425
\pinlabel {2} [c] at 718 360
\pinlabel {3} [c] at 730 440
\pinlabel {4} [c] at 863 355
\pinlabel {5} [c] at 1010 358
\pinlabel {6} [c] at 1000 440
\pinlabel {7} [c] at 912 423
\pinlabel {1} [c] at 1413 425
\pinlabel {2} [c] at 1313 360
\pinlabel {3} [c] at 1325 440
\pinlabel {5} [c] at 1605 358
\pinlabel {4} [c] at 1595 440
\pinlabel {6} [c] at 1507 423
\endlabellist
\centerline{\includegraphics[width=5.5in]{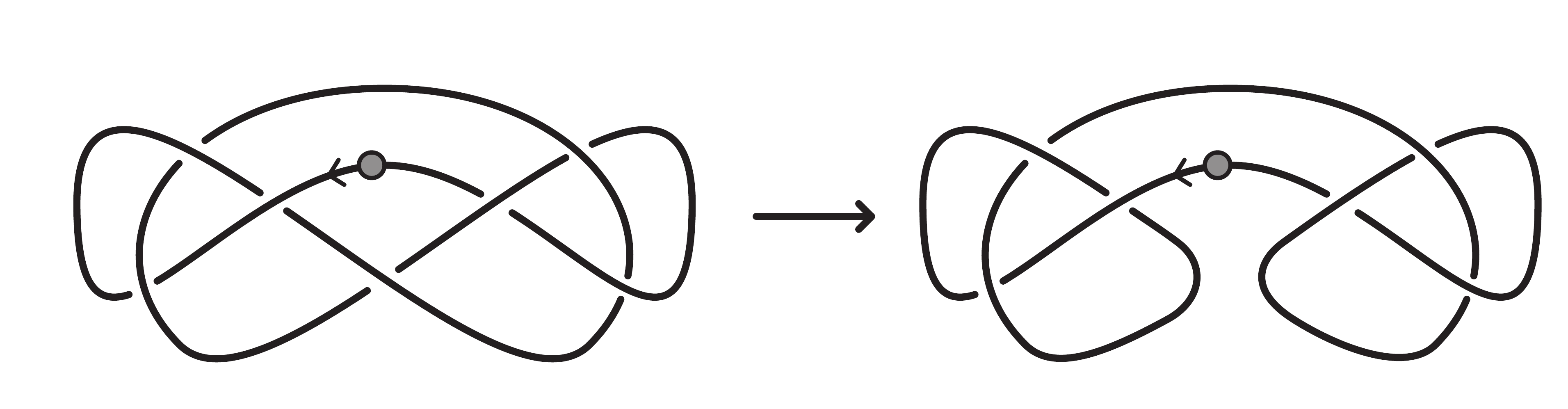}}
\caption{Given the diagram of the knot $7_7$ with Gauss code $(1,2,3,1,4,5,6,3,2,4,7,6,5,7)$, splicing at crossing $4$ gives the diagram of $3_1\#3_1$ with Gauss code 
$(1,2,3,1,2,3,4,5,6,4,5,6)$.}
\label{Fi:GaussSplice}
\end{figure}

The next step is to construct a dictionary $D_{\text{splice}}$ in which one can look up any prime alternating knot $K$, say with crossing number $n$, and find $n$ lists of knot types, where each list describes the connect sum decomposition of the knot which results from splicing a given diagram for $K$ (the one described by its imported Gauss code) at one of its $n$ crossings. 

Recall that we have used our imported data to construct a list $D_G$ of dictionaries $D_G[n]$ which give us, for every prime alternating knot $K$ with crossing number $n$, the reduced Gauss code $G$ of some reduced alternating diagram $D$ of $K$ (and a list of the signs of the crossings in $D$). Given any $i=1,\hdots, n$, let $c=c_i$. We can write $G=w_1cw_2cw_3$, where $w_2$ is nonempty, as is at least one of $w_1$ or $w_3$. After relabeling, $w_1\overline{w_2}w_3$ is a Gauss code for the diagram obtained from $D$ via a $u^-$ type splice at $c$; $\overline{w_2}$ denotes the reverse of $w_2$. Let $G_i$ be the reduced form of this Gauss code.  

The Gauss codes $G_1,\hdots,G_n$ constructed in this way from $G$ are the reduced Gauss codes which describe the knot diagrams which result from each of the possible $u^-$ type splices on $D$.  
For each $i=1,\hdots, n$, decompose $G_i$ into its connect summands, as described in \textsection\ref{S:GBasics}.  Then compute the reduced Gauss code of each summand, look up the associated DT code in $D_{\text{DT}}$, and record the knot type. 
%
%
For example, for knots with seven crossings, $D_{\text{splice}}$ looks like:\\

\centerline{\begin{tabular}{||c|ccccccc||}\hline\hline
knot&splice&splice&splice&splice&splice&splice&splice\\ \hline
$7_1$&$0_1$&$0_1$&$0_1$&$0_1$&$0_1$&$0_1$&$0_1$\\ \hline
$7_2$&$6_1$&$5_1$&$5_1$&$6_1$&$6_1$&$6_1$&$6_1$\\ \hline
$7_3$&$6_1$&$3_1$&$3_1$&$3_1$&$3_1$&$6_1$&$6_1$\\ \hline
$7_4$&$6_2$&$3_1,3_1$&$6_2$&$6_2$&$6_2$&$6_2$&$6_2$\\ \hline
$7_5$&$6_2$&$5_2$&$5_2$&$4_1$&$4_1$&$4_1$&$6_2$\\ \hline
$7_6$&$6_1$&$5_2$&$5_2$&$6_2$&$6_2$&$6_3$&$6_3$\\ \hline
$7_7$&$6_2$&$6_3$&$6_3$&$3_1,3_1$&$6_3$&$6_3$&$6_2$\\ \hline\hline
\end{tabular}}



\subsection{Crosscap numbers from splice data}\label{S:GCc}

Finally, we are ready to construct a list $D_{u^-}$ of dictionaries $D_{u^-}[n]$, each listing $u^-(K)$ for the unknot and all prime alternating knots $K$ with $n$ crossings. Because all prime alternating knots $K$ satisfy $u^-(K)=cc(K)$ by Theorem \ref{T:main}, we can then copy these dictionaries to obtain the list $D_{\text{cc}}$ of dictionaries $D_{\text{cc}}[n]$ recording the crosscap numbers of all prime alternating knots with $n$ crossings, for $n\geq 3$.

First, let $D_{u^-}[0]=\left\{\text{\textquoteleft}
0_1\text{\textquoteright}:0\right\}$, with $D_{u^-}[1]=[{}]=D_{u^-}[2]$. Then starting with crossing number $n=3$ and increasing from there, compute $D_{u^-}[n]$ as follows. For each $K$ in $D_{\text{splice}}[n]$ and each $i=1,\hdots, n$, consider $D_{\text{splice}}[n][K][i]=[K'_{i,1},\hdots,K'_{i,m_i}]$. Each $K'_{i,j}$ has fewer crossings $n_{i,j}$ than $K$, so we can look up each $D_{u^-}[n_{i,j}][K'_{i,j}]$. This gives:
\[cc(K)=u^-(K)=D_{u^-}[n](K)=1+\min_{i=1,\hdots,n}\sum_{j=1}^{m_i}u^-(K'_{i,j})=1+\min_{i=1,\hdots,n}\sum_{j=1}^{m_i}D_{u^-}[n_{i,j}][K'_{i,j}].\]
In other words, we build the dictionary $D_{u^-}$ of splice-unknotting numbers inductively, by looking at the connect summands of the diagrams obtained by $u^-$-splices on a given diagram, looking up these summands' crosscap numbers in $D_{u^-}$, summing, minimizing, and adding 1.

\section*{Appendix: tables of crosscap numbers}

\subsection*{Crosscap numbers $n=cc(K)$ of 11-crossing prime alternating knots $K$}

\centerline{\begin{tabular}{||c|c||c|c||c|c||c|c||c|c||c|c||c|c||c|c||}\hline\hline
$K$&$n$&$K$&$n$&$K$&$n$&$K$&$n$&$K$&$n$&$K$&$n$&$K$&$n$&$K$&$n$\\ \hline
$11_{1}$&$5$
&$11_{2}$&$5$
&$11_{3}$&$5$
&$11_{4}$&$4$
&$11_{5}$&$5$
&$11_{6}$&$5$
&$11_{7}$&$4$
&$11_{8}$&$4$\\ \hline
$11_{9}$&$3$
&$11_{10}$&$4$
&$11_{11}$&$5$
&$11_{12}$&$5$
&$11_{13}$&$4$
&$11_{14}$&$5$
&$11_{15}$&$4$
&$11_{16}$&$5$\\ \hline
$11_{17}$&$5$
&$11_{18}$&$5$
&$11_{19}$&$5$
&$11_{20}$&$5$
&$11_{21}$&$4$
&$11_{22}$&$4$
&$11_{23}$&$5$
&$11_{24}$&$5$\\ \hline
$11_{25}$&$5$
&$11_{26}$&$5$
&$11_{27}$&$5$
&$11_{28}$&$5$
&$11_{29}$&$4$
&$11_{30}$&$5$
&$11_{31}$&$5$
&$11_{32}$&$5$\\ \hline
$11_{33}$&$4$
&$11_{34}$&$5$
&$11_{35}$&$5$
&$11_{36}$&$5$
&$11_{37}$&$4$
&$11_{38}$&$5$
&$11_{39}$&$4$
&$11_{40}$&$4$\\ \hline
$11_{41}$&$5$
&$11_{42}$&$5$
&$11_{43}$&$5$
&$11_{44}$&$5$
&$11_{45}$&$4$
&$11_{46}$&$4$
&$11_{47}$&$5$
&$11_{48}$&$5$\\ \hline
$11_{49}$&$5$
&$11_{50}$&$4$
&$11_{51}$&$5$
&$11_{52}$&$5$
&$11_{53}$&$4$
&$11_{54}$&$5$
&$11_{55}$&$4$
&$11_{56}$&$5$\\ \hline
$11_{57}$&$4$
&$11_{58}$&$4$
&$11_{59}$&$3$
&$11_{60}$&$4$
&$11_{61}$&$4$
&$11_{62}$&$3$
&$11_{63}$&$4$
&$11_{64}$&$5$\\ \hline
$11_{65}$&$4$
&$11_{66}$&$5$
&$11_{67}$&$4$
&$11_{68}$&$4$
&$11_{69}$&$5$
&$11_{70}$&$5$
&$11_{71}$&$5$
&$11_{72}$&$5$\\ \hline
$11_{73}$&$5$
&$11_{74}$&$3$
&$11_{75}$&$4$
&$11_{76}$&$5$
&$11_{77}$&$5$
&$11_{78}$&$5$
&$11_{79}$&$5$
&$11_{80}$&$5$\\ \hline
$11_{81}$&$4$
&$11_{82}$&$4$
&$11_{83}$&$4$
&$11_{84}$&$5$
&$11_{85}$&$5$
&$11_{86}$&$4$
&$11_{87}$&$5$
&$11_{88}$&$4$\\ \hline
$11_{89}$&$5$
&$11_{90}$&$4$
&$11_{91}$&$5$
&$11_{92}$&$4$
&$11_{93}$&$4$
&$11_{94}$&$5$
&$11_{95}$&$4$
&$11_{96}$&$5$\\ \hline
$11_{97}$&$3$
&$11_{98}$&$4$
&$11_{99}$&$4$
&$11_{100}$&$5$
&$11_{101}$&$5$
&$11_{102}$&$4$
&$11_{103}$&$4$
&$11_{104}$&$5$\\ \hline
$11_{105}$&$5$
&$11_{106}$&$4$
&$11_{107}$&$4$
&$11_{108}$&$4$
&$11_{109}$&$5$
&$11_{110}$&$4$
&$11_{111}$&$4$
&$11_{112}$&$5$\\ \hline
$11_{113}$&$4$
&$11_{114}$&$5$
&$11_{115}$&$4$
&$11_{116}$&$5$
&$11_{117}$&$5$
&$11_{118}$&$4$
&$11_{119}$&$4$
&$11_{120}$&$5$\\ \hline
$11_{121}$&$5$
&$11_{122}$&$5$
&$11_{123}$&$4$
&$11_{124}$&$5$
&$11_{125}$&$5$
&$11_{126}$&$5$
&$11_{127}$&$4$
&$11_{128}$&$5$\\ \hline
$11_{129}$&$4$
&$11_{130}$&$5$
&$11_{131}$&$5$
&$11_{132}$&$5$
&$11_{133}$&$4$
&$11_{134}$&$5$
&$11_{135}$&$5$
&$11_{136}$&$5$\\ \hline
$11_{137}$&$4$
&$11_{138}$&$5$
&$11_{139}$&$4$
&$11_{140}$&$3$
&$11_{141}$&$4$
&$11_{142}$&$3$
&$11_{143}$&$4$
&$11_{144}$&$4$\\ \hline
$11_{145}$&$4$
&$11_{146}$&$5$
&$11_{147}$&$5$
&$11_{148}$&$4$
&$11_{149}$&$5$
&$11_{150}$&$5$
&$11_{151}$&$5$
&$11_{152}$&$4$\\ \hline
$11_{153}$&$4$
&$11_{154}$&$4$
&$11_{155}$&$5$
&$11_{156}$&$4$
&$11_{157}$&$5$
&$11_{158}$&$4$
&$11_{159}$&$5$
&$11_{160}$&$5$\\ \hline
$11_{161}$&$3$
&$11_{162}$&$5$
&$11_{163}$&$4$
&$11_{164}$&$5$
&$11_{165}$&$4$
&$11_{166}$&$3$
&$11_{167}$&$5$
&$11_{168}$&$5$\\ \hline
$11_{169}$&$4$
&$11_{170}$&$5$
&$11_{171}$&$5$
&$11_{172}$&$5$
&$11_{173}$&$5$
&$11_{174}$&$4$
&$11_{175}$&$5$
&$11_{176}$&$5$\\ \hline
$11_{177}$&$4$
&$11_{178}$&$5$
&$11_{179}$&$3$
&$11_{180}$&$4$
&$11_{181}$&$4$
&$11_{182}$&$4$
&$11_{183}$&$5$
&$11_{184}$&$4$\\ \hline
$11_{185}$&$4$
&$11_{186}$&$5$
&$11_{187}$&$5$
&$11_{188}$&$3$
&$11_{189}$&$5$
&$11_{190}$&$4$
&$11_{191}$&$4$
&$11_{192}$&$4$\\ \hline
$11_{193}$&$4$
&$11_{194}$&$4$
&$11_{195}$&$3$
&$11_{196}$&$5$
&$11_{197}$&$5$
&$11_{198}$&$4$
&$11_{199}$&$4$
&$11_{200}$&$4$\\ \hline
$11_{201}$&$4$
&$11_{202}$&$5$
&$11_{203}$&$3$
&$11_{204}$&$4$
&$11_{205}$&$4$
&$11_{206}$&$3$
&$11_{207}$&$4$
&$11_{208}$&$5$\\ \hline
$11_{209}$&$5$
&$11_{210}$&$4$
&$11_{211}$&$4$
&$11_{212}$&$5$
&$11_{213}$&$5$
&$11_{214}$&$4$
&$11_{215}$&$4$
&$11_{216}$&$5$\\ \hline
$11_{217}$&$5$
&$11_{218}$&$5$
&$11_{219}$&$4$
&$11_{220}$&$4$
&$11_{221}$&$4$
&$11_{222}$&$4$
&$11_{223}$&$3$
&$11_{224}$&$4$\\ \hline
$11_{225}$&$3$
&$11_{226}$&$4$
&$11_{227}$&$5$
&$11_{228}$&$5$
&$11_{229}$&$4$
&$11_{230}$&$3$
&$11_{231}$&$4$
&$11_{232}$&$4$\\ \hline
\end{tabular}}
\centerline{\begin{tabular}{||c|c||c|c||c|c||c|c||c|c||c|c||c|c||c|c||}\hline
$K$&$n$&$K$&$n$&$K$&$n$&$K$&$n$&$K$&$n$&$K$&$n$&$K$&$n$&$K$&$n$\\ \hline
$11_{233}$&$5$
&$11_{234}$&$3$
&$11_{235}$&$4$
&$11_{236}$&$5$
&$11_{237}$&$4$
&$11_{238}$&$4$
&$11_{239}$&$5$
&$11_{240}$&$3$\\ \hline
$11_{241}$&$4$
&$11_{242}$&$3$
&$11_{243}$&$4$
&$11_{244}$&$5$
&$11_{245}$&$4$
&$11_{246}$&$3$
&$11_{247}$&$2$
&$11_{248}$&$5$\\ \hline
$11_{249}$&$4$
&$11_{250}$&$3$
&$11_{251}$&$5$
&$11_{252}$&$4$
&$11_{253}$&$5$
&$11_{254}$&$4$
&$11_{255}$&$5$
&$11_{256}$&$4$\\ \hline
$11_{257}$&$4$
&$11_{258}$&$3$
&$11_{259}$&$3$
&$11_{260}$&$3$
&$11_{261}$&$4$
&$11_{262}$&$4$
&$11_{263}$&$3$
&$11_{264}$&$5$\\ \hline
$11_{265}$&$4$
&$11_{266}$&$5$
&$11_{267}$&$5$
&$11_{268}$&$4$
&$11_{269}$&$4$
&$11_{270}$&$5$
&$11_{271}$&$5$
&$11_{272}$&$5$\\ \hline
$11_{273}$&$5$
&$11_{274}$&$5$
&$11_{275}$&$5$
&$11_{276}$&$5$
&$11_{277}$&$5$
&$11_{278}$&$4$
&$11_{279}$&$3$
&$11_{280}$&$4$\\ \hline
$11_{281}$&$4$
&$11_{282}$&$4$
&$11_{283}$&$5$
&$11_{284}$&$5$
&$11_{285}$&$5$
&$11_{286}$&$4$
&$11_{287}$&$5$
&$11_{288}$&$5$\\ \hline
$11_{289}$&$5$
&$11_{290}$&$4$
&$11_{291}$&$4$
&$11_{292}$&$5$
&$11_{293}$&$3$
&$11_{294}$&$4$
&$11_{295}$&$4$
&$11_{296}$&$4$\\ \hline
$11_{297}$&$5$
&$11_{298}$&$5$
&$11_{299}$&$4$
&$11_{300}$&$5$
&$11_{301}$&$5$
&$11_{302}$&$4$
&$11_{303}$&$4$
&$11_{304}$&$4$\\ \hline
$11_{305}$&$4$
&$11_{306}$&$4$
&$11_{307}$&$4$
&$11_{308}$&$3$
&$11_{309}$&$4$
&$11_{310}$&$3$
&$11_{311}$&$4$
&$11_{312}$&$4$\\ \hline
$11_{313}$&$3$
&$11_{314}$&$5$
&$11_{315}$&$5$
&$11_{316}$&$4$
&$11_{317}$&$4$
&$11_{318}$&$5$
&$11_{319}$&$5$
&$11_{320}$&$4$\\ \hline
$11_{321}$&$5$
&$11_{322}$&$5$
&$11_{323}$&$3$
&$11_{324}$&$4$
&$11_{325}$&$4$
&$11_{326}$&$5$
&$11_{327}$&$5$
&$11_{328}$&$5$\\ \hline
$11_{329}$&$5$
&$11_{330}$&$3$
&$11_{331}$&$4$
&$11_{332}$&$5$
&$11_{333}$&$3$
&$11_{334}$&$3$
&$11_{335}$&$4$
&$11_{336}$&$3$\\ \hline
$11_{337}$&$4$
&$11_{338}$&$3$
&$11_{339}$&$3$
&$11_{340}$&$4$
&$11_{341}$&$3$
&$11_{342}$&$2$
&$11_{343}$&$3$
&$11_{344}$&$5$\\ \hline
$11_{345}$&$4$
&$11_{346}$&$3$
&$11_{347}$&$4$
&$11_{348}$&$4$
&$11_{349}$&$5$
&$11_{350}$&$5$
&$11_{351}$&$5$
&$11_{352}$&$4$\\ \hline
$11_{353}$&$5$
&$11_{354}$&$4$
&$11_{355}$&$3$
&$11_{356}$&$4$
&$11_{357}$&$4$
&$11_{358}$&$2$
&$11_{359}$&$3$
&$11_{360}$&$3$\\ \hline
$11_{361}$&$3$
&$11_{362}$&$3$
&$11_{363}$&$3$
&$11_{364}$&$2$
&$11_{365}$&$3$
&$11_{366}$&$4$
&$11_{367}$&$1$
&&\\ \hline\hline
\end{tabular}}

\subsection*{Crosscap numbers $n=cc(K)$ of 12-crossing, prime alternating knots $K$}

\centerline{\begin{tabular}{||c|c||c|c||c|c||c|c||c|c||c|c||c|c||c|c||c|c||c|c||c|c||c|c||c|c||c|c||c|c||c|c||c|c||c|c||}\hline\hline
$K$&$n$&$K$&$n$&$K$&$n$&$K$&$n$&$K$&$n$&$K$&$n$&$K$&$n$&$K$&$n$\\ \hline
$12_{1}$&$5$
&$12_{2}$&$4$
&$12_{3}$&$5$
&$12_{4}$&$6$
&$12_{5}$&$6$
&$12_{6}$&$5$
&$12_{7}$&$6$
&$12_{8}$&$5$\\ \hline
$12_{9}$&$4$
&$12_{10}$&$6$
&$12_{11}$&$5$
&$12_{12}$&$5$
&$12_{13}$&$5$
&$12_{14}$&$6$
&$12_{15}$&$5$
&$12_{16}$&$5$\\ \hline
$12_{17}$&$5$
&$12_{18}$&$4$
&$12_{19}$&$5$
&$12_{20}$&$5$
&$12_{21}$&$6$
&$12_{22}$&$4$
&$12_{23}$&$5$
&$12_{24}$&$4$\\ \hline
$12_{25}$&$6$
&$12_{26}$&$5$
&$12_{27}$&$5$
&$12_{28}$&$6$
&$12_{29}$&$6$
&$12_{30}$&$5$
&$12_{31}$&$4$
&$12_{32}$&$5$\\ \hline
$12_{33}$&$5$
&$12_{34}$&$4$
&$12_{35}$&$5$
&$12_{36}$&$4$
&$12_{37}$&$5$
&$12_{38}$&$4$
&$12_{39}$&$5$
&$12_{40}$&$6$\\ \hline
$12_{41}$&$5$
&$12_{42}$&$5$
&$12_{43}$&$6$
&$12_{44}$&$6$
&$12_{45}$&$5$
&$12_{46}$&$5$
&$12_{47}$&$6$
&$12_{48}$&$6$\\ \hline
$12_{49}$&$5$
&$12_{50}$&$5$
&$12_{51}$&$5$
&$12_{52}$&$4$
&$12_{53}$&$5$
&$12_{54}$&$5$
&$12_{55}$&$5$
&$12_{56}$&$4$\\ \hline
$12_{57}$&$5$
&$12_{58}$&$6$
&$12_{59}$&$6$
&$12_{60}$&$6$
&$12_{61}$&$5$
&$12_{62}$&$5$
&$12_{63}$&$6$
&$12_{64}$&$6$\\ \hline
$12_{65}$&$5$
&$12_{66}$&$5$
&$12_{67}$&$5$
&$12_{68}$&$5$
&$12_{69}$&$6$
&$12_{70}$&$5$
&$12_{71}$&$5$
&$12_{72}$&$5$\\ \hline
$12_{73}$&$6$
&$12_{74}$&$6$
&$12_{75}$&$5$
&$12_{76}$&$4$
&$12_{77}$&$6$
&$12_{78}$&$5$
&$12_{79}$&$5$
&$12_{80}$&$5$\\ \hline
$12_{81}$&$5$
&$12_{82}$&$6$
&$12_{83}$&$6$
&$12_{84}$&$5$
&$12_{85}$&$5$
&$12_{86}$&$5$
&$12_{87}$&$5$
&$12_{88}$&$6$\\ \hline
$12_{89}$&$5$
&$12_{90}$&$6$
&$12_{91}$&$5$
&$12_{92}$&$5$
&$12_{93}$&$4$
&$12_{94}$&$5$
&$12_{95}$&$5$
&$12_{96}$&$4$\\ \hline
$12_{97}$&$4$
&$12_{98}$&$5$
&$12_{99}$&$6$
&$12_{100}$&$5$
&$12_{101}$&$5$
&$12_{102}$&$6$
&$12_{103}$&$6$
&$12_{104}$&$5$\\ \hline
\end{tabular}}
\centerline{\begin{tabular}{||c|c||c|c||c|c||c|c||c|c||c|c||c|c||c|c||c|c||c|c||c|c||c|c||c|c||c|c||c|c||c|c||c|c||c|c||}\hline
$K$&$n$&$K$&$n$&$K$&$n$&$K$&$n$&$K$&$n$&$K$&$n$&$K$&$n$&$K$&$n$\\ \hline
$12_{105}$&$4$
&$12_{106}$&$5$
&$12_{107}$&$6$
&$12_{108}$&$6$
&$12_{109}$&$5$
&$12_{110}$&$5$
&$12_{111}$&$5$
&$12_{112}$&$5$\\ \hline
$12_{113}$&$6$
&$12_{114}$&$6$
&$12_{115}$&$5$
&$12_{116}$&$5$
&$12_{117}$&$6$
&$12_{118}$&$5$
&$12_{119}$&$5$
&$12_{120}$&$6$\\ \hline
$12_{121}$&$5$
&$12_{122}$&$5$
&$12_{123}$&$4$
&$12_{124}$&$5$
&$12_{125}$&$6$
&$12_{126}$&$6$
&$12_{127}$&$5$
&$12_{128}$&$4$\\ \hline
$12_{129}$&$5$
&$12_{130}$&$5$
&$12_{131}$&$5$
&$12_{132}$&$6$
&$12_{133}$&$5$
&$12_{134}$&$5$
&$12_{135}$&$5$
&$12_{136}$&$5$\\ \hline
$12_{137}$&$5$
&$12_{138}$&$5$
&$12_{139}$&$6$
&$12_{140}$&$5$
&$12_{141}$&$5$
&$12_{142}$&$5$
&$12_{143}$&$4$
&$12_{144}$&$5$\\ \hline
$12_{145}$&$5$
&$12_{146}$&$3$
&$12_{147}$&$4$
&$12_{148}$&$5$
&$12_{149}$&$6$
&$12_{150}$&$5$
&$12_{151}$&$5$
&$12_{152}$&$4$\\ \hline
$12_{153}$&$4$
&$12_{154}$&$6$
&$12_{155}$&$5$
&$12_{156}$&$4$
&$12_{157}$&$5$
&$12_{158}$&$4$
&$12_{159}$&$5$
&$12_{160}$&$4$\\ \hline
$12_{161}$&$5$
&$12_{162}$&$6$
&$12_{163}$&$5$
&$12_{164}$&$5$
&$12_{165}$&$4$
&$12_{166}$&$5$
&$12_{167}$&$5$
&$12_{168}$&$4$\\ \hline
$12_{169}$&$3$
&$12_{170}$&$5$
&$12_{171}$&$5$
&$12_{172}$&$4$
&$12_{173}$&$5$
&$12_{174}$&$5$
&$12_{175}$&$6$
&$12_{176}$&$4$\\ \hline
$12_{177}$&$5$
&$12_{178}$&$4$
&$12_{179}$&$5$
&$12_{180}$&$5$
&$12_{181}$&$6$
&$12_{182}$&$5$
&$12_{183}$&$4$
&$12_{184}$&$6$\\ \hline
$12_{185}$&$5$
&$12_{186}$&$5$
&$12_{187}$&$5$
&$12_{188}$&$5$
&$12_{189}$&$5$
&$12_{190}$&$5$
&$12_{191}$&$6$
&$12_{192}$&$5$\\ \hline
$12_{193}$&$4$
&$12_{194}$&$5$
&$12_{195}$&$4$
&$12_{196}$&$5$
&$12_{197}$&$4$
&$12_{198}$&$6$
&$12_{199}$&$6$
&$12_{200}$&$5$\\ \hline
$12_{201}$&$4$
&$12_{202}$&$5$
&$12_{203}$&$5$
&$12_{204}$&$5$
&$12_{205}$&$4$
&$12_{206}$&$4$
&$12_{207}$&$4$
&$12_{208}$&$5$\\ \hline
$12_{209}$&$6$
&$12_{210}$&$6$
&$12_{211}$&$5$
&$12_{212}$&$5$
&$12_{213}$&$5$
&$12_{214}$&$6$
&$12_{215}$&$5$
&$12_{216}$&$4$\\ \hline
$12_{217}$&$5$
&$12_{218}$&$5$
&$12_{219}$&$5$
&$12_{220}$&$5$
&$12_{221}$&$5$
&$12_{222}$&$6$
&$12_{223}$&$4$
&$12_{224}$&$5$\\ \hline
$12_{225}$&$6$
&$12_{226}$&$6$
&$12_{227}$&$6$
&$12_{228}$&$5$
&$12_{229}$&$6$
&$12_{230}$&$5$
&$12_{231}$&$5$
&$12_{232}$&$6$\\ \hline
$12_{233}$&$6$
&$12_{234}$&$5$
&$12_{235}$&$5$
&$12_{236}$&$4$
&$12_{237}$&$5$
&$12_{238}$&$5$
&$12_{239}$&$4$
&$12_{240}$&$5$\\ \hline
$12_{241}$&$5$
&$12_{242}$&$5$
&$12_{243}$&$5$
&$12_{244}$&$5$
&$12_{245}$&$5$
&$12_{246}$&$4$
&$12_{247}$&$5$
&$12_{248}$&$4$\\ \hline
$12_{249}$&$5$
&$12_{250}$&$4$
&$12_{251}$&$5$
&$12_{252}$&$4$
&$12_{253}$&$5$
&$12_{254}$&$4$
&$12_{255}$&$4$
&$12_{256}$&$5$\\ \hline
$12_{257}$&$6$
&$12_{258}$&$5$
&$12_{259}$&$4$
&$12_{260}$&$4$
&$12_{261}$&$5$
&$12_{262}$&$4$
&$12_{263}$&$6$
&$12_{264}$&$5$\\ \hline
$12_{265}$&$6$
&$12_{266}$&$5$
&$12_{267}$&$5$
&$12_{268}$&$6$
&$12_{269}$&$5$
&$12_{270}$&$4$
&$12_{271}$&$5$
&$12_{272}$&$5$\\ \hline
$12_{273}$&$6$
&$12_{274}$&$5$
&$12_{275}$&$5$
&$12_{276}$&$4$
&$12_{277}$&$5$
&$12_{278}$&$5$
&$12_{279}$&$5$
&$12_{280}$&$5$\\ \hline
$12_{281}$&$5$
&$12_{282}$&$6$
&$12_{283}$&$5$
&$12_{284}$&$5$
&$12_{285}$&$5$
&$12_{286}$&$5$
&$12_{287}$&$6$
&$12_{288}$&$6$\\ \hline
$12_{289}$&$5$
&$12_{290}$&$5$
&$12_{291}$&$4$
&$12_{292}$&$5$
&$12_{293}$&$6$
&$12_{294}$&$5$
&$12_{295}$&$6$
&$12_{296}$&$6$\\ \hline
$12_{297}$&$5$
&$12_{298}$&$5$
&$12_{299}$&$4$
&$12_{300}$&$5$
&$12_{301}$&$5$
&$12_{302}$&$5$
&$12_{303}$&$5$
&$12_{304}$&$4$\\ \hline
$12_{305}$&$5$
&$12_{306}$&$5$
&$12_{307}$&$5$
&$12_{308}$&$5$
&$12_{309}$&$5$
&$12_{310}$&$6$
&$12_{311}$&$5$
&$12_{312}$&$4$\\ \hline
$12_{313}$&$5$
&$12_{314}$&$6$
&$12_{315}$&$6$
&$12_{316}$&$6$
&$12_{317}$&$5$
&$12_{318}$&$5$
&$12_{319}$&$5$
&$12_{320}$&$4$\\ \hline
$12_{321}$&$4$
&$12_{322}$&$5$
&$12_{323}$&$6$
&$12_{324}$&$5$
&$12_{325}$&$5$
&$12_{326}$&$5$
&$12_{327}$&$5$
&$12_{328}$&$5$\\ \hline
$12_{329}$&$5$
&$12_{330}$&$4$
&$12_{331}$&$5$
&$12_{332}$&$5$
&$12_{333}$&$6$
&$12_{334}$&$5$
&$12_{335}$&$5$
&$12_{336}$&$6$\\ \hline
$12_{337}$&$6$
&$12_{338}$&$6$
&$12_{339}$&$4$
&$12_{340}$&$6$
&$12_{341}$&$6$
&$12_{342}$&$6$
&$12_{343}$&$5$
&$12_{344}$&$5$\\ \hline
$12_{345}$&$4$
&$12_{346}$&$5$
&$12_{347}$&$5$
&$12_{348}$&$6$
&$12_{349}$&$5$
&$12_{350}$&$6$
&$12_{351}$&$5$
&$12_{352}$&$6$\\ \hline
$12_{353}$&$5$
&$12_{354}$&$5$
&$12_{355}$&$4$
&$12_{356}$&$4$
&$12_{357}$&$5$
&$12_{358}$&$5$
&$12_{359}$&$6$
&$12_{360}$&$5$\\ \hline
$12_{361}$&$6$
&$12_{362}$&$5$
&$12_{363}$&$5$
&$12_{364}$&$6$
&$12_{365}$&$4$
&$12_{366}$&$5$
&$12_{367}$&$4$
&$12_{368}$&$5$\\ \hline
$12_{369}$&$3$
&$12_{370}$&$4$
&$12_{371}$&$4$
&$12_{372}$&$5$
&$12_{373}$&$4$
&$12_{374}$&$5$
&$12_{375}$&$4$
&$12_{376}$&$4$\\ \hline
$12_{377}$&$5$
&$12_{378}$&$4$
&$12_{379}$&$3$
&$12_{380}$&$3$
&$12_{381}$&$5$
&$12_{382}$&$4$
&$12_{383}$&$5$
&$12_{384}$&$5$\\ \hline
$12_{385}$&$5$
&$12_{386}$&$5$
&$12_{387}$&$5$
&$12_{388}$&$6$
&$12_{389}$&$6$
&$12_{390}$&$6$
&$12_{391}$&$5$
&$12_{392}$&$4$\\ \hline
$12_{393}$&$6$
&$12_{394}$&$5$
&$12_{395}$&$5$
&$12_{396}$&$5$
&$12_{397}$&$5$
&$12_{398}$&$4$
&$12_{399}$&$5$
&$12_{400}$&$5$\\ \hline
\end{tabular}}
\centerline{\begin{tabular}{||c|c||c|c||c|c||c|c||c|c||c|c||c|c||c|c||c|c||c|c||c|c||c|c||c|c||c|c||c|c||c|c||c|c||c|c||}\hline
$K$&$n$&$K$&$n$&$K$&$n$&$K$&$n$&$K$&$n$&$K$&$n$&$K$&$n$&$K$&$n$\\ \hline
$12_{401}$&$5$
&$12_{402}$&$5$
&$12_{403}$&$5$
&$12_{404}$&$4$
&$12_{405}$&$5$
&$12_{406}$&$6$
&$12_{407}$&$5$
&$12_{408}$&$6$\\ \hline
$12_{409}$&$4$
&$12_{410}$&$5$
&$12_{411}$&$5$
&$12_{412}$&$5$
&$12_{413}$&$5$
&$12_{414}$&$4$
&$12_{415}$&$6$
&$12_{416}$&$5$\\ \hline
$12_{417}$&$6$
&$12_{418}$&$5$
&$12_{419}$&$5$
&$12_{420}$&$4$
&$12_{421}$&$4$
&$12_{422}$&$3$
&$12_{423}$&$4$
&$12_{424}$&$5$\\ \hline
$12_{425}$&$4$
&$12_{426}$&$6$
&$12_{427}$&$6$
&$12_{428}$&$5$
&$12_{429}$&$5$
&$12_{430}$&$5$
&$12_{431}$&$6$
&$12_{432}$&$6$\\ \hline
$12_{433}$&$6$
&$12_{434}$&$5$
&$12_{435}$&$6$
&$12_{436}$&$4$
&$12_{437}$&$5$
&$12_{438}$&$5$
&$12_{439}$&$6$
&$12_{440}$&$5$\\ \hline
$12_{441}$&$5$
&$12_{442}$&$4$
&$12_{443}$&$4$
&$12_{444}$&$4$
&$12_{445}$&$5$
&$12_{446}$&$5$
&$12_{447}$&$4$
&$12_{448}$&$4$\\ \hline
$12_{449}$&$5$
&$12_{450}$&$5$
&$12_{451}$&$5$
&$12_{452}$&$6$
&$12_{453}$&$5$
&$12_{454}$&$4$
&$12_{455}$&$5$
&$12_{456}$&$6$\\ \hline
$12_{457}$&$5$
&$12_{458}$&$6$
&$12_{459}$&$5$
&$12_{460}$&$6$
&$12_{461}$&$6$
&$12_{462}$&$5$
&$12_{463}$&$4$
&$12_{464}$&$5$\\ \hline
$12_{465}$&$6$
&$12_{466}$&$5$
&$12_{467}$&$6$
&$12_{468}$&$5$
&$12_{469}$&$5$
&$12_{470}$&$6$
&$12_{471}$&$4$
&$12_{472}$&$6$\\ \hline
$12_{473}$&$5$
&$12_{474}$&$6$
&$12_{475}$&$5$
&$12_{476}$&$4$
&$12_{477}$&$6$
&$12_{478}$&$5$
&$12_{479}$&$5$
&$12_{480}$&$6$\\ \hline
$12_{481}$&$4$
&$12_{482}$&$4$
&$12_{483}$&$6$
&$12_{484}$&$6$
&$12_{485}$&$5$
&$12_{486}$&$6$
&$12_{487}$&$6$
&$12_{488}$&$4$\\ \hline
$12_{489}$&$5$
&$12_{490}$&$5$
&$12_{491}$&$5$
&$12_{492}$&$5$
&$12_{493}$&$4$
&$12_{494}$&$5$
&$12_{495}$&$5$
&$12_{496}$&$6$\\ \hline
$12_{497}$&$6$
&$12_{498}$&$5$
&$12_{499}$&$6$
&$12_{500}$&$5$
&$12_{501}$&$5$
&$12_{502}$&$4$
&$12_{503}$&$4$
&$12_{504}$&$5$\\ \hline
$12_{505}$&$5$
&$12_{506}$&$5$
&$12_{507}$&$4$
&$12_{508}$&$5$
&$12_{509}$&$6$
&$12_{510}$&$6$
&$12_{511}$&$5$
&$12_{512}$&$5$\\ \hline
$12_{513}$&$5$
&$12_{514}$&$5$
&$12_{515}$&$5$
&$12_{516}$&$6$
&$12_{517}$&$4$
&$12_{518}$&$5$
&$12_{519}$&$4$
&$12_{520}$&$4$\\ \hline
$12_{521}$&$4$
&$12_{522}$&$5$
&$12_{523}$&$5$
&$12_{524}$&$5$
&$12_{525}$&$4$
&$12_{526}$&$6$
&$12_{527}$&$5$
&$12_{528}$&$5$\\ \hline
$12_{529}$&$5$
&$12_{530}$&$5$
&$12_{531}$&$5$
&$12_{532}$&$4$
&$12_{533}$&$5$
&$12_{534}$&$5$
&$12_{535}$&$5$
&$12_{536}$&$4$\\ \hline
$12_{537}$&$5$
&$12_{538}$&$4$
&$12_{539}$&$5$
&$12_{540}$&$5$
&$12_{541}$&$4$
&$12_{542}$&$4$
&$12_{543}$&$6$
&$12_{544}$&$5$\\ \hline
$12_{545}$&$5$
&$12_{546}$&$6$
&$12_{547}$&$6$
&$12_{548}$&$4$
&$12_{549}$&$4$
&$12_{550}$&$5$
&$12_{551}$&$4$
&$12_{552}$&$4$\\ \hline
$12_{553}$&$5$
&$12_{554}$&$5$
&$12_{555}$&$5$
&$12_{556}$&$5$
&$12_{557}$&$4$
&$12_{558}$&$5$
&$12_{559}$&$5$
&$12_{560}$&$5$\\ \hline
$12_{561}$&$5$
&$12_{562}$&$5$
&$12_{563}$&$4$
&$12_{564}$&$4$
&$12_{565}$&$5$
&$12_{566}$&$5$
&$12_{567}$&$5$
&$12_{568}$&$5$\\ \hline
$12_{569}$&$5$
&$12_{570}$&$5$
&$12_{571}$&$6$
&$12_{572}$&$5$
&$12_{573}$&$4$
&$12_{574}$&$4$
&$12_{575}$&$5$
&$12_{576}$&$3$\\ \hline
$12_{577}$&$4$
&$12_{578}$&$5$
&$12_{579}$&$5$
&$12_{580}$&$3$
&$12_{581}$&$4$
&$12_{582}$&$4$
&$12_{583}$&$5$
&$12_{584}$&$5$\\ \hline
$12_{585}$&$5$
&$12_{586}$&$5$
&$12_{587}$&$4$
&$12_{588}$&$6$
&$12_{589}$&$5$
&$12_{590}$&$4$
&$12_{591}$&$4$
&$12_{592}$&$6$\\ \hline
$12_{593}$&$5$
&$12_{594}$&$4$
&$12_{595}$&$4$
&$12_{596}$&$3$
&$12_{597}$&$4$
&$12_{598}$&$5$
&$12_{599}$&$5$
&$12_{600}$&$4$\\ \hline
$12_{601}$&$4$
&$12_{602}$&$5$
&$12_{603}$&$5$
&$12_{604}$&$6$
&$12_{605}$&$4$
&$12_{606}$&$5$
&$12_{607}$&$5$
&$12_{608}$&$5$\\ \hline
$12_{609}$&$5$
&$12_{610}$&$4$
&$12_{611}$&$6$
&$12_{612}$&$4$
&$12_{613}$&$5$
&$12_{614}$&$6$
&$12_{615}$&$6$
&$12_{616}$&$5$\\ \hline
$12_{617}$&$5$
&$12_{618}$&$5$
&$12_{619}$&$4$
&$12_{620}$&$5$
&$12_{621}$&$5$
&$12_{622}$&$5$
&$12_{623}$&$5$
&$12_{624}$&$5$\\ \hline
$12_{625}$&$5$
&$12_{626}$&$6$
&$12_{627}$&$6$
&$12_{628}$&$5$
&$12_{629}$&$6$
&$12_{630}$&$5$
&$12_{631}$&$6$
&$12_{632}$&$4$\\ \hline
$12_{633}$&$5$
&$12_{634}$&$4$
&$12_{635}$&$5$
&$12_{636}$&$3$
&$12_{637}$&$5$
&$12_{638}$&$5$
&$12_{639}$&$5$
&$12_{640}$&$4$\\ \hline
$12_{641}$&$3$
&$12_{642}$&$4$
&$12_{643}$&$4$
&$12_{644}$&$4$
&$12_{645}$&$6$
&$12_{646}$&$5$
&$12_{647}$&$4$
&$12_{648}$&$5$\\ \hline
$12_{649}$&$4$
&$12_{650}$&$5$
&$12_{651}$&$4$
&$12_{652}$&$5$
&$12_{653}$&$4$
&$12_{654}$&$5$
&$12_{655}$&$5$
&$12_{656}$&$5$\\ \hline
$12_{657}$&$5$
&$12_{658}$&$5$
&$12_{659}$&$6$
&$12_{660}$&$4$
&$12_{661}$&$5$
&$12_{662}$&$6$
&$12_{663}$&$4$
&$12_{664}$&$4$\\ \hline
$12_{665}$&$5$
&$12_{666}$&$6$
&$12_{667}$&$4$
&$12_{668}$&$5$
&$12_{669}$&$3$
&$12_{670}$&$5$
&$12_{671}$&$5$
&$12_{672}$&$6$\\ \hline
$12_{673}$&$5$
&$12_{674}$&$6$
&$12_{675}$&$5$
&$12_{676}$&$5$
&$12_{677}$&$5$
&$12_{678}$&$5$
&$12_{679}$&$4$
&$12_{680}$&$5$\\ \hline
$12_{681}$&$4$
&$12_{682}$&$4$
&$12_{683}$&$4$
&$12_{684}$&$5$
&$12_{685}$&$6$
&$12_{686}$&$6$
&$12_{687}$&$6$
&$12_{688}$&$5$\\ \hline
$12_{689}$&$4$
&$12_{690}$&$4$
&$12_{691}$&$4$
&$12_{692}$&$5$
&$12_{693}$&$4$
&$12_{694}$&$4$
&$12_{695}$&$6$
&$12_{696}$&$5$\\ \hline
\end{tabular}}
\centerline{\begin{tabular}{||c|c||c|c||c|c||c|c||c|c||c|c||c|c||c|c||c|c||c|c||c|c||c|c||c|c||c|c||c|c||c|c||c|c||c|c||}\hline
$K$&$n$&$K$&$n$&$K$&$n$&$K$&$n$&$K$&$n$&$K$&$n$&$K$&$n$&$K$&$n$\\ \hline
$12_{697}$&$6$
&$12_{698}$&$5$
&$12_{699}$&$5$
&$12_{700}$&$5$
&$12_{701}$&$5$
&$12_{702}$&$4$
&$12_{703}$&$6$
&$12_{704}$&$5$\\ \hline
$12_{705}$&$6$
&$12_{706}$&$5$
&$12_{707}$&$5$
&$12_{708}$&$4$
&$12_{709}$&$5$
&$12_{710}$&$6$
&$12_{711}$&$5$
&$12_{712}$&$6$\\ \hline
$12_{713}$&$5$
&$12_{714}$&$4$
&$12_{715}$&$5$
&$12_{716}$&$3$
&$12_{717}$&$4$
&$12_{718}$&$5$
&$12_{719}$&$5$
&$12_{720}$&$4$\\ \hline
$12_{721}$&$5$
&$12_{722}$&$2$
&$12_{723}$&$3$
&$12_{724}$&$4$
&$12_{725}$&$4$
&$12_{726}$&$4$
&$12_{727}$&$5$
&$12_{728}$&$5$\\ \hline
$12_{729}$&$5$
&$12_{730}$&$5$
&$12_{731}$&$4$
&$12_{732}$&$4$
&$12_{733}$&$3$
&$12_{734}$&$6$
&$12_{735}$&$4$
&$12_{736}$&$5$\\ \hline
$12_{737}$&$5$
&$12_{738}$&$4$
&$12_{739}$&$4$
&$12_{740}$&$4$
&$12_{741}$&$6$
&$12_{742}$&$4$
&$12_{743}$&$4$
&$12_{744}$&$3$\\ \hline
$12_{745}$&$3$
&$12_{746}$&$5$
&$12_{747}$&$5$
&$12_{748}$&$4$
&$12_{749}$&$4$
&$12_{750}$&$4$
&$12_{751}$&$5$
&$12_{752}$&$4$\\ \hline
$12_{753}$&$3$
&$12_{754}$&$5$
&$12_{755}$&$6$
&$12_{756}$&$5$
&$12_{757}$&$5$
&$12_{758}$&$4$
&$12_{759}$&$3$
&$12_{760}$&$4$\\ \hline
$12_{761}$&$5$
&$12_{762}$&$3$
&$12_{763}$&$4$
&$12_{764}$&$5$
&$12_{765}$&$6$
&$12_{766}$&$5$
&$12_{767}$&$4$
&$12_{768}$&$5$\\ \hline
$12_{769}$&$5$
&$12_{770}$&$6$
&$12_{771}$&$5$
&$12_{772}$&$4$
&$12_{773}$&$4$
&$12_{774}$&$4$
&$12_{775}$&$4$
&$12_{776}$&$5$\\ \hline
$12_{777}$&$4$
&$12_{778}$&$6$
&$12_{779}$&$5$
&$12_{780}$&$6$
&$12_{781}$&$5$
&$12_{782}$&$4$
&$12_{783}$&$5$
&$12_{784}$&$5$\\ \hline
$12_{785}$&$5$
&$12_{786}$&$5$
&$12_{787}$&$4$
&$12_{788}$&$6$
&$12_{789}$&$4$
&$12_{790}$&$5$
&$12_{791}$&$3$
&$12_{792}$&$4$\\ \hline
$12_{793}$&$6$
&$12_{794}$&$4$
&$12_{795}$&$5$
&$12_{796}$&$3$
&$12_{797}$&$4$
&$12_{798}$&$6$
&$12_{799}$&$5$
&$12_{800}$&$4$\\ \hline
$12_{801}$&$4$
&$12_{802}$&$3$
&$12_{803}$&$2$
&$12_{804}$&$5$
&$12_{805}$&$4$
&$12_{806}$&$5$
&$12_{807}$&$5$
&$12_{808}$&$4$\\ \hline
$12_{809}$&$5$
&$12_{810}$&$5$
&$12_{811}$&$4$
&$12_{812}$&$5$
&$12_{813}$&$4$
&$12_{814}$&$5$
&$12_{815}$&$4$
&$12_{816}$&$5$\\ \hline
$12_{817}$&$4$
&$12_{818}$&$4$
&$12_{819}$&$5$
&$12_{820}$&$4$
&$12_{821}$&$5$
&$12_{822}$&$4$
&$12_{823}$&$4$
&$12_{824}$&$4$\\ \hline
$12_{825}$&$4$
&$12_{826}$&$3$
&$12_{827}$&$3$
&$12_{828}$&$5$
&$12_{829}$&$4$
&$12_{830}$&$5$
&$12_{831}$&$5$
&$12_{832}$&$4$\\ \hline
$12_{833}$&$4$
&$12_{834}$&$4$
&$12_{835}$&$3$
&$12_{836}$&$4$
&$12_{837}$&$4$
&$12_{838}$&$2$
&$12_{839}$&$3$
&$12_{840}$&$4$\\ \hline
$12_{841}$&$4$
&$12_{842}$&$3$
&$12_{843}$&$3$
&$12_{844}$&$5$
&$12_{845}$&$3$
&$12_{846}$&$5$
&$12_{847}$&$4$
&$12_{848}$&$5$\\ \hline
$12_{849}$&$5$
&$12_{850}$&$4$
&$12_{851}$&$5$
&$12_{852}$&$5$
&$12_{853}$&$4$
&$12_{854}$&$4$
&$12_{855}$&$4$
&$12_{856}$&$5$\\ \hline
$12_{857}$&$5$
&$12_{858}$&$4$
&$12_{859}$&$3$
&$12_{860}$&$4$
&$12_{861}$&$5$
&$12_{862}$&$5$
&$12_{863}$&$4$
&$12_{864}$&$5$\\ \hline
$12_{865}$&$5$
&$12_{866}$&$6$
&$12_{867}$&$6$
&$12_{868}$&$6$
&$12_{869}$&$4$
&$12_{870}$&$5$
&$12_{871}$&$5$
&$12_{872}$&$4$\\ \hline
$12_{873}$&$4$
&$12_{874}$&$6$
&$12_{875}$&$5$
&$12_{876}$&$4$
&$12_{877}$&$4$
&$12_{878}$&$3$
&$12_{879}$&$4$
&$12_{880}$&$5$\\ \hline
$12_{881}$&$3$
&$12_{882}$&$4$
&$12_{883}$&$4$
&$12_{884}$&$5$
&$12_{885}$&$5$
&$12_{886}$&$5$
&$12_{887}$&$6$
&$12_{888}$&$5$\\ \hline
$12_{889}$&$4$
&$12_{890}$&$5$
&$12_{891}$&$5$
&$12_{892}$&$5$
&$12_{893}$&$6$
&$12_{894}$&$5$
&$12_{895}$&$6$
&$12_{896}$&$4$\\ \hline
$12_{897}$&$5$
&$12_{898}$&$5$
&$12_{899}$&$5$
&$12_{900}$&$6$
&$12_{901}$&$5$
&$12_{902}$&$5$
&$12_{903}$&$6$
&$12_{904}$&$5$\\ \hline
$12_{905}$&$4$
&$12_{906}$&$6$
&$12_{907}$&$5$
&$12_{908}$&$5$
&$12_{909}$&$4$
&$12_{910}$&$5$
&$12_{911}$&$5$
&$12_{912}$&$4$\\ \hline
$12_{913}$&$4$
&$12_{914}$&$5$
&$12_{915}$&$6$
&$12_{916}$&$5$
&$12_{917}$&$5$
&$12_{918}$&$5$
&$12_{919}$&$5$
&$12_{920}$&$4$\\ \hline
$12_{921}$&$5$
&$12_{922}$&$6$
&$12_{923}$&$4$
&$12_{924}$&$5$
&$12_{925}$&$5$
&$12_{926}$&$4$
&$12_{927}$&$4$
&$12_{928}$&$5$\\ \hline
$12_{929}$&$4$
&$12_{930}$&$4$
&$12_{931}$&$5$
&$12_{932}$&$4$
&$12_{933}$&$5$
&$12_{934}$&$6$
&$12_{935}$&$5$
&$12_{936}$&$5$\\ \hline
$12_{937}$&$3$
&$12_{938}$&$4$
&$12_{939}$&$5$
&$12_{940}$&$4$
&$12_{941}$&$4$
&$12_{942}$&$4$
&$12_{943}$&$5$
&$12_{944}$&$5$\\ \hline
$12_{945}$&$5$
&$12_{946}$&$4$
&$12_{947}$&$4$
&$12_{948}$&$5$
&$12_{949}$&$5$
&$12_{950}$&$4$
&$12_{951}$&$5$
&$12_{952}$&$4$\\ \hline
$12_{953}$&$5$
&$12_{954}$&$5$
&$12_{955}$&$4$
&$12_{956}$&$5$
&$12_{957}$&$5$
&$12_{958}$&$5$
&$12_{959}$&$5$
&$12_{960}$&$6$\\ \hline
$12_{961}$&$6$
&$12_{962}$&$5$
&$12_{963}$&$4$
&$12_{964}$&$5$
&$12_{965}$&$6$
&$12_{966}$&$5$
&$12_{967}$&$5$
&$12_{968}$&$5$\\ \hline
$12_{969}$&$4$
&$12_{970}$&$3$
&$12_{971}$&$4$
&$12_{972}$&$4$
&$12_{973}$&$5$
&$12_{974}$&$5$
&$12_{975}$&$4$
&$12_{976}$&$5$\\ \hline
$12_{977}$&$4$
&$12_{978}$&$4$
&$12_{979}$&$5$
&$12_{980}$&$5$
&$12_{981}$&$4$
&$12_{982}$&$5$
&$12_{983}$&$5$
&$12_{984}$&$3$\\ \hline
$12_{985}$&$4$
&$12_{986}$&$5$
&$12_{987}$&$5$
&$12_{988}$&$4$
&$12_{989}$&$5$
&$12_{990}$&$5$
&$12_{991}$&$4$
&$12_{992}$&$6$\\ \hline
\end{tabular}}
\centerline{\begin{tabular}{||c|c||c|c||c|c||c|c||c|c||c|c||c|c||c|c||c|c||c|c||c|c||c|c||c|c||c|c||c|c||c|c||c|c||c|c||}\hline
$K$&$n$&$K$&$n$&$K$&$n$&$K$&$n$&$K$&$n$&$K$&$n$&$K$&$n$&$K$&$n$\\ \hline
$12_{993}$&$5$
&$12_{994}$&$6$
&$12_{995}$&$5$
&$12_{996}$&$5$
&$12_{997}$&$5$
&$12_{998}$&$6$
&$12_{999}$&$5$
&$12_{1000}$&$4$\\ \hline
$12_{1001}$&$4$
&$12_{1002}$&$5$
&$12_{1003}$&$5$
&$12_{1004}$&$6$
&$12_{1005}$&$5$
&$12_{1006}$&$5$
&$12_{1007}$&$4$
&$12_{1008}$&$5$\\ \hline
$12_{1009}$&$4$
&$12_{1010}$&$5$
&$12_{1011}$&$4$
&$12_{1012}$&$4$
&$12_{1013}$&$4$
&$12_{1014}$&$5$
&$12_{1015}$&$4$
&$12_{1016}$&$5$\\ \hline
$12_{1017}$&$3$
&$12_{1018}$&$4$
&$12_{1019}$&$6$
&$12_{1020}$&$5$
&$12_{1021}$&$6$
&$12_{1022}$&$5$
&$12_{1023}$&$4$
&$12_{1024}$&$4$\\ \hline
$12_{1025}$&$5$
&$12_{1026}$&$4$
&$12_{1027}$&$3$
&$12_{1028}$&$4$
&$12_{1029}$&$3$
&$12_{1030}$&$3$
&$12_{1031}$&$3$
&$12_{1032}$&$4$\\ \hline
$12_{1033}$&$4$
&$12_{1034}$&$4$
&$12_{1035}$&$4$
&$12_{1036}$&$4$
&$12_{1037}$&$5$
&$12_{1038}$&$5$
&$12_{1039}$&$4$
&$12_{1040}$&$4$\\ \hline
$12_{1041}$&$5$
&$12_{1042}$&$5$
&$12_{1043}$&$5$
&$12_{1044}$&$5$
&$12_{1045}$&$4$
&$12_{1046}$&$5$
&$12_{1047}$&$5$
&$12_{1048}$&$5$\\ \hline
$12_{1049}$&$5$
&$12_{1050}$&$5$
&$12_{1051}$&$4$
&$12_{1052}$&$5$
&$12_{1053}$&$5$
&$12_{1054}$&$5$
&$12_{1055}$&$5$
&$12_{1056}$&$6$\\ \hline
$12_{1057}$&$5$
&$12_{1058}$&$5$
&$12_{1059}$&$4$
&$12_{1060}$&$5$
&$12_{1061}$&$6$
&$12_{1062}$&$4$
&$12_{1063}$&$4$
&$12_{1064}$&$5$\\ \hline
$12_{1065}$&$5$
&$12_{1066}$&$5$
&$12_{1067}$&$6$
&$12_{1068}$&$4$
&$12_{1069}$&$6$
&$12_{1070}$&$5$
&$12_{1071}$&$5$
&$12_{1072}$&$5$\\ \hline
$12_{1073}$&$5$
&$12_{1074}$&$4$
&$12_{1075}$&$4$
&$12_{1076}$&$6$
&$12_{1077}$&$5$
&$12_{1078}$&$5$
&$12_{1079}$&$6$
&$12_{1080}$&$4$\\ \hline
$12_{1081}$&$5$
&$12_{1082}$&$4$
&$12_{1083}$&$4$
&$12_{1084}$&$4$
&$12_{1085}$&$5$
&$12_{1086}$&$5$
&$12_{1087}$&$5$
&$12_{1088}$&$6$\\ \hline
$12_{1089}$&$4$
&$12_{1090}$&$5$
&$12_{1091}$&$5$
&$12_{1092}$&$5$
&$12_{1093}$&$5$
&$12_{1094}$&$4$
&$12_{1095}$&$3$
&$12_{1096}$&$5$\\ \hline
$12_{1097}$&$5$
&$12_{1098}$&$6$
&$12_{1099}$&$5$
&$12_{1100}$&$5$
&$12_{1101}$&$5$
&$12_{1102}$&$6$
&$12_{1103}$&$5$
&$12_{1104}$&$5$\\ \hline
$12_{1105}$&$6$
&$12_{1106}$&$4$
&$12_{1107}$&$3$
&$12_{1108}$&$4$
&$12_{1109}$&$5$
&$12_{1110}$&$5$
&$12_{1111}$&$4$
&$12_{1112}$&$5$\\ \hline
$12_{1113}$&$5$
&$12_{1114}$&$3$
&$12_{1115}$&$4$
&$12_{1116}$&$5$
&$12_{1117}$&$6$
&$12_{1118}$&$4$
&$12_{1119}$&$5$
&$12_{1120}$&$4$\\ \hline
$12_{1121}$&$5$
&$12_{1122}$&$5$
&$12_{1123}$&$6$
&$12_{1124}$&$6$
&$12_{1125}$&$4$
&$12_{1126}$&$4$
&$12_{1127}$&$4$
&$12_{1128}$&$3$\\ \hline
$12_{1129}$&$4$
&$12_{1130}$&$4$
&$12_{1131}$&$3$
&$12_{1132}$&$4$
&$12_{1133}$&$5$
&$12_{1134}$&$3$
&$12_{1135}$&$4$
&$12_{1136}$&$5$\\ \hline
$12_{1137}$&$4$
&$12_{1138}$&$3$
&$12_{1139}$&$4$
&$12_{1140}$&$4$
&$12_{1141}$&$5$
&$12_{1142}$&$3$
&$12_{1143}$&$5$
&$12_{1144}$&$4$\\ \hline
$12_{1145}$&$3$
&$12_{1146}$&$4$
&$12_{1147}$&$4$
&$12_{1148}$&$3$
&$12_{1149}$&$2$
&$12_{1150}$&$5$
&$12_{1151}$&$4$
&$12_{1152}$&$6$\\ \hline
$12_{1153}$&$4$
&$12_{1154}$&$5$
&$12_{1155}$&$6$
&$12_{1156}$&$4$
&$12_{1157}$&$2$
&$12_{1158}$&$3$
&$12_{1159}$&$4$
&$12_{1160}$&$4$\\ \hline
$12_{1161}$&$3$
&$12_{1162}$&$3$
&$12_{1163}$&$4$
&$12_{1164}$&$4$
&$12_{1165}$&$3$
&$12_{1166}$&$3$
&$12_{1167}$&$6$
&$12_{1168}$&$5$\\ \hline
$12_{1169}$&$4$
&$12_{1170}$&$4$
&$12_{1171}$&$3$
&$12_{1172}$&$5$
&$12_{1173}$&$5$
&$12_{1174}$&$4$
&$12_{1175}$&$5$
&$12_{1176}$&$4$\\ \hline
$12_{1177}$&$5$
&$12_{1178}$&$4$
&$12_{1179}$&$3$
&$12_{1180}$&$5$
&$12_{1181}$&$4$
&$12_{1182}$&$5$
&$12_{1183}$&$4$
&$12_{1184}$&$5$\\ \hline
$12_{1185}$&$5$
&$12_{1186}$&$5$
&$12_{1187}$&$6$
&$12_{1188}$&$6$
&$12_{1189}$&$5$
&$12_{1190}$&$5$
&$12_{1191}$&$4$
&$12_{1192}$&$5$\\ \hline
$12_{1193}$&$6$
&$12_{1194}$&$4$
&$12_{1195}$&$5$
&$12_{1196}$&$5$
&$12_{1197}$&$5$
&$12_{1198}$&$5$
&$12_{1199}$&$5$
&$12_{1200}$&$4$\\ \hline
$12_{1201}$&$5$
&$12_{1202}$&$5$
&$12_{1203}$&$4$
&$12_{1204}$&$4$
&$12_{1205}$&$3$
&$12_{1206}$&$6$
&$12_{1207}$&$5$
&$12_{1208}$&$5$\\ \hline
$12_{1209}$&$5$
&$12_{1210}$&$4$
&$12_{1211}$&$6$
&$12_{1212}$&$5$
&$12_{1213}$&$5$
&$12_{1214}$&$2$
&$12_{1215}$&$4$
&$12_{1216}$&$4$\\ \hline
$12_{1217}$&$5$
&$12_{1218}$&$4$
&$12_{1219}$&$4$
&$12_{1220}$&$3$
&$12_{1221}$&$5$
&$12_{1222}$&$5$
&$12_{1223}$&$4$
&$12_{1224}$&$4$\\ \hline
$12_{1225}$&$6$
&$12_{1226}$&$4$
&$12_{1227}$&$5$
&$12_{1228}$&$5$
&$12_{1229}$&$6$
&$12_{1230}$&$5$
&$12_{1231}$&$5$
&$12_{1232}$&$4$\\ \hline
$12_{1233}$&$3$
&$12_{1234}$&$4$
&$12_{1235}$&$4$
&$12_{1236}$&$4$
&$12_{1237}$&$5$
&$12_{1238}$&$4$
&$12_{1239}$&$5$
&$12_{1240}$&$3$\\ \hline
$12_{1241}$&$4$
&$12_{1242}$&$2$
&$12_{1243}$&$3$
&$12_{1244}$&$4$
&$12_{1245}$&$5$
&$12_{1246}$&$4$
&$12_{1247}$&$3$
&$12_{1248}$&$5$\\ \hline
$12_{1249}$&$6$
&$12_{1250}$&$5$
&$12_{1251}$&$6$
&$12_{1252}$&$6$
&$12_{1253}$&$5$
&$12_{1254}$&$4$
&$12_{1255}$&$4$
&$12_{1256}$&$4$\\ \hline
$12_{1257}$&$5$
&$12_{1258}$&$5$
&$12_{1259}$&$4$
&$12_{1260}$&$5$
&$12_{1261}$&$5$
&$12_{1262}$&$4$
&$12_{1263}$&$5$
&$12_{1264}$&$4$\\ \hline
$12_{1265}$&$5$
&$12_{1266}$&$4$
&$12_{1267}$&$4$
&$12_{1268}$&$5$
&$12_{1269}$&$5$
&$12_{1270}$&$6$
&$12_{1271}$&$5$
&$12_{1272}$&$5$\\ \hline
$12_{1273}$&$3$
&$12_{1274}$&$4$
&$12_{1275}$&$5$
&$12_{1276}$&$3$
&$12_{1277}$&$4$
&$12_{1278}$&$2$
&$12_{1279}$&$3$
&$12_{1280}$&$6$\\ \hline
$12_{1281}$&$4$
&$12_{1282}$&$3$
&$12_{1283}$&$3$
&$12_{1284}$&$4$
&$12_{1285}$&$3$
&$12_{1286}$&$2$
&$12_{1287}$&$3$
&$12_{1288}$&$4$\\ \hline
\end{tabular}}

\end{document}